\documentclass[a4paper,11pt] {article}
\usepackage{geometry}
\usepackage[latin1]{inputenc}
\usepackage[T1]{fontenc}
\usepackage[latin1]{inputenc}
\usepackage{lmodern}
\usepackage[english]{babel}
\usepackage{xcolor}
\usepackage{bm}
\usepackage{tikz}

\newcommand{\deltat}{{\Delta t}}
\newcommand{\aleq}{\stackrel{<}{\sim}}

\usepackage{amsthm}
\newtheorem{df}{Definition}[section]

\newtheorem{cor}{Corollary}[section]
\usepackage{amsmath}
\usepackage{amssymb}
\usepackage{graphicx}
\newcommand{\mesh}{\mathcal{T}}
\newcommand{\R}{\mathbb{R}}

\newcommand{\stik}{\sum_{\substack{\sigma \in {\cal{E}}(K) \\ \sigma=K|L}}}

\newcommand{\bu}{\boldsymbol{u}}
\newcommand{\bn}{\boldsymbol{n}}

\newcommand{\bv}{\boldsymbol{v}}

\newcommand{\bU}{\boldsymbol{U}}
\newcommand{\bV}{\boldsymbol{V}}

\newcommand{\Det}{\Delta t}

\newcommand{\dS}{\, {\mathrm d}S}
\newcommand{\dx}{\, {\mathrm d}x}

\catcode`\@=11 \@addtoreset{equation}{section}

\catcode`\@=12

\newtheorem{Theorem}{Theorem}[section]
\newtheorem{Proposition}{Proposition}[section]
\newtheorem{Lemma}{Lemma}[section]
\newtheorem{Corollary}{Corollary}[section]
\newtheorem{Remark}{Remark}[section]

\newcommand{\bTheorem}[1]{
\begin{Theorem} \label{T#1} }
\newcommand{\eT}{\end{Theorem}}

\newcommand{\bProposition}[1]{
\begin{Proposition} \label{P#1}}
\newcommand{\eP}{\end{Proposition}}

\newcommand{\bLemma}[1]{
\begin{Lemma} \label{L#1} }
\newcommand{\eL}{\end{Lemma}}

\newcommand{\bCorollary}[1]{
\begin{Corollary} \label{C#1} }
\newcommand{\eC}{\end{Corollary}}

\newcommand{\bFormula}[1]{
\begin{equation} \label{#1}}
\newcommand{\eF}{\end{equation}}

\newcommand{\bRemark}[1]{
\begin{Remark} \label{R#1} }
\newcommand{\eR}{\end{Remark}}

\newcommand{\Ov}[1]{\overline{#1}}

\newcommand{\vr}{\varrho}

\newcommand{\vc}[1]{{\bf #1}}
\newcommand{\Div}{{\rm div}_x}
\newcommand{\Grad}{\nabla_x}

\newcommand{\tn}[1]{\mbox {\F #1}}

\newcommand{\Rm}{\mbox{\FF R}}

\newcommand{\bProof}{{\bf Proof: }}

\newcommand{\vu}{\vc{u}}

\newcommand{\ep}{\varepsilon}

\font\F=msbm10 scaled 1000

\font\FF=msbm10 scaled 800
\definecolor{grey}{rgb}{0.85,0.85,0.85}
\date{}

\long\def\greybox#1{%
    \newbox\contentbox%
    \newbox\bkgdbox%
    \setbox\contentbox\hbox to \hsize{%
        \vtop{
            \kern\columnsep
            \hbox to \hsize{%
                \kern\columnsep%
                \advance\hsize by -2\columnsep%
                \setlength{\textwidth}{\hsize}%
                \vbox{
                    \parskip=\baselineskip
                    \parindent=0bp
                    #1
                }%
                \kern\columnsep%
            }%
            \kern\columnsep%
        }%
    }%
    \setbox\bkgdbox\vbox{
        \color{grey}
        \hrule width  \wd\contentbox %
               height \ht\contentbox %
               depth  \dp\contentbox
        \color{black}
    }%
    \wd\bkgdbox=0bp%
    \vbox{\hbox to \hsize{\box\bkgdbox\box\contentbox}}%
    \vskip\baselineskip%
}

\DeclareMathOperator{\dv}{div}

\DeclareMathOperator{\dt}{dt}

\DeclareMathOperator{\intt}{int}
\DeclareMathOperator{\extt}{ext}

\numberwithin{equation}{section}

\title{Error estimates for a numerical method for the compressible Navier-Stokes system on sufficiently smooth domains}

\author{Eduard Feireisl \thanks{The research of E.F. leading to these results has received funding from the European Research Council under the European Union's Seventh Framework
Programme (FP7/2007-2013)/ ERC Grant Agreement 320078. The Institute of Mathematics of the Academy of Sciences of the Czech
Republic is supported by RVO:67985840.}  \and Radim Ho\v sek \thanks{The research of R.H. leading to these results has received funding from the European Research Council under the European Union's Seventh Framework
Programme (FP7/2007-2013)/ ERC Grant Agreement 320078. The Institute of Mathematics of the Academy of Sciences of the Czech
Republic is supported by RVO:67985840.} \and David Maltese \and Anton\' \i n Novotn\' y\thanks{The work
of D.M. and A.N. has been supported by the MODTERCOM project within the APEX programme of the Provence-Alpes-C\^ote d'Azur region}}

\begin{document}

\maketitle

\bigskip

\centerline{Institute of Mathematics of the Academy of Sciences of the Czech Republic}

\centerline{\v Zitn\' a 25, CZ-115 67 Praha 1, Czech Republic}

\medskip
\centerline{Institut Math\'ematiques de Toulon, EA2134, University of Toulon}

\centerline{BP 20132, 839 57 La Garde, France }

\begin{abstract}
We derive an {\it a priori} error estimate for the numerical solution   obtained by time and space discretization by the finite volume/finite element method  of the barotropic Navier--Stokes equations.
The numerical solution  on a convenient polyhedral domain approximating a sufficiently smooth bounded domain is compared with an exact solution of the barotropic Navier--Stokes equations with  a bounded density. The
result is unconditional in the sense that there are no assumed bounds on the numerical solution. It is obtained
by the combination of discrete relative energy inequality derived in \cite{GHMN} and several recent results in the theory of compressible Navier-Stokes equations concerning blow up criterion established in \cite{SuWaZha} and weak strong uniqueness principle established in \cite{FeJiNo}.

\end{abstract}



\section{Introduction}
\label{i}

We consider the compressible Navier-Stokes equations in the barotropic regime in a space-time cylinder $Q_T=(0,T)\times\Omega$, where $T>0$ is arbitrarily large and $\Omega \subset R^3$ is a bounded domain:

\bFormula{i1}
\partial_t \vr + \Div (\vr \vu) = 0,
\eF
\bFormula{i2}
\partial_t (\vr \vu) + \Div (\vr \vu \otimes \vu) + \Grad p(\vr) = \Div \tn{S}(\Grad \vu),
\eF
In equations (\ref{i1}--\ref{i2}) $\vr=\vr(t,x)\ge 0$ and $\vc u=\vc u(t,x)\in R^3$, $t\in [0,T)$, $x\in \Omega$ are unknown density and velocity fields, while $\tn S$ and $p$ are viscous stress and pressure characterizing the fluid
via the constitutive relations

\bFormula{i3}
\tn{S} (\Grad \vu)  = \mu \left( \Grad \vu + \Grad^t \vu - \frac{2}{3} \Div \vu \tn{I} \right) ,\  \mu > 0,
\eF

\bFormula{i4}
p \in C^2(0,\infty) \cap C^1[0, \infty),\
p(0) = 0 , \ p'(\vr) > 0 \ \mbox{for all}\ \vr \geq 0, \ \lim_{\vr \to \infty} \frac{p'(\vr)}{\vr^{\gamma-1}} = p_\infty > 0,
\eF
where $\gamma\ge 1$.

Assumption $p'(0)>0$ excludes constitutive laws behaving as $\vr^\gamma$ as $\vr \to 0^+$. Error estimate stated in Theorem \ref{M1} however still holds in the case $\lim_{\vr \to 0^+} \frac{p'(\vr)}{\vr^{\gamma-1}}>0$ at the price of some additional difficulties, see \cite{GHMN} for more details.

Equations (\ref{i1}--\ref{i2}) are completed with the no-slip boundary conditions

\bFormula{i6}
\vu |_{\partial \Omega} = 0,
\eF
and initial conditions
\bFormula{i7}
\vr(0, \cdot) = \vr_0, \ \vu(0, \cdot) = \vu_0, \ \vr_0 > 0 \ \mbox{in} \ \Ov{\Omega}.
\eF
We notice that under assumption (\ref{i3}), we may write
\bFormula{i5}
\Div \tn{S}(\Grad \vu) = \mu \Delta \vu + \frac \mu 3 \Grad \Div \vu.
\eF

 The results on error estimates for numerical schemes for the compressible Navier-Stokes equations are in the mathematical literature on short supply. We refer the reader to papers of Liu \cite{Liu1}, \cite{Liu2},
Yovanovic \cite{YOVAN}, Gallouet et al. \cite{GHMN}.

In \cite{GHMN} the authors have developed a methodology of deriving unconditional error estimates for the numerical schemes to the compressible Navier-Stokes equations (\ref{i1}--\ref{i7}) and applied it to the numerical scheme
(\ref{num1}--\ref{num3}) discretizing the system on  polyhedral domains. They have obtained error estimates for the discrete solution with respect to a {\it classical solution} of the system on the same (polyhedral) domain.
In spite of the fact that \cite{GHMN} provides the first and   to the best of our knowledge so far the sole error estimate for discrete solutions of a finite volume/finite element approximation to a model of compressible fluids that does not need any assumed bounds on the numerical solution itself, it has two weak points: 1) The existence of classical solutions on at least a short time interval to the compressible Navier-Stokes equations is known for smooth $C^3$ domains (see Valli, Zajaczkowski \cite{VAZA} or Cho, Choe, Kim \cite{ChoChoeKim}) but may not be in general true on the polyhedral domains. 2) The numerical solutions are compared with the classical exact solutions (as is usual in any previous existing mathematical literature). In this paper we address both points raised above and to a certain extent remove the limitations of the theory presented in \cite{GHMN}.

More precisely, we generalize the result of Gallouet et al.  \cite[Theorem 3.1]{GHMN} in two directions:
\begin{description}
\item {(1)} The physical domain $\Omega$ filled by the fluid and the numerical domain $\Omega_h$, $h>0$ approximating the physical domain do not need to coincide.
\item {(2)} If the physical domain is sufficiently smooth (at least of class $C^3$)  and the $C^3-$ initial data satisfy natural compatibility conditions, we are able to obtain the unconditional error estimates with respect to any {\it weak exact solution with bounded density}.
    \end{description}
    As in \cite{GHMN}, and in contrast with any other error estimate literature dealing with finite volume or mixed finite volume/finite element methods for compressible fluids
     (Yovanovich \cite{YOVAN}, Canc\`es et al \cite{CancesMathisSeguin2014Relative}, Eymard et al. \cite{EGGH98}, Villa, Villedieu \cite{VV03}, Rohde, Yovanovich \cite{RJ05}, Gastaldo et al. \cite{GHLT} and others) this result does not require any  assumed bounds on the discrete solution: the sole bounds needed for the result are those provided by the numerical scheme. Moreover, in contrast with \cite{GHMN} and with all above mentioned papers, the exact solution is solely weak solution with bounded density. This seemingly weak hypothesis is compensated by the regularity and compatibility conditions imposed on initial data that make possible a (sophisticated) bootstrapping argument showing that weak solutions with bounded density are in fact strong solutions in the class investigated in \cite{GHMN}.

     These results are achieved by using the following tools:
     \begin{description}
     \item {(1)} The technique introduced in \cite{GHMN} modified in order to accommodate non-zero velocity of the exact sample solution on the boundary of the numerical domain.
     \item   {(2)} Three fundamental recent results from the theory of compressible Navier-Stokes equations, namely
         \begin{itemize}
         \item Local in time existence of strong solutions in class (\ref{r7}--\ref{r8}) by Cho, Choe, Kim \cite{ChoChoeKim}.
             \item Weak strong uniqueness principle proved in \cite{FeJiNo} (see also \cite{FENOSU}).
             \item Blow up criterion for strong solutions in the class (\ref{r7}--\ref{r8}) by Sun, Wang, Zhang
             \cite{SuWaZha}.
         \end{itemize}
         The three above mentioned items allow to show that the weak solution with bounded density emanating from the sufficiently smooth initial data is
         in fact a strong solution defined on the large time interval $[0,T)$.
             \item {(3)} Bootstrapping  argument using recent results on maximal regularity for parabolic systems by Danchin \cite{DANCHIN}, Denk, Pruess, Hieber \cite{DEHIEPR} and Krylov \cite{Krylov}.
                 The last item allows to bootstrap the strong solution in the class Cho, Choe, Kim \cite{ChoChoeKim} to the class needed for the error estimates in \cite{GHMN}, provided
                 a certain compatibility condition for the initial data is satisfied.
\end{description}

\section{Preliminaries}
\subsection{Weak and strong solutions to the Navier-Stokes system}
\label{r}

We introduce the notion of the weak solution to system (\ref{i1}--\ref{i4}):

\begin{df}[Weak solutions]\label{ws}
{ Let $\varrho_0  : \Omega \to [0, +\infty) $ and $\bu_0 :\Omega \to  \R^3$  with  finite energy $E_0=\int_\Omega (\frac{1}{2} \varrho_0 |\bu_0|^2 + {{H}}(\vr_0)) \dx$ and finite mass $0<M_0=\int_\Omega\vr_0\dx$.}
We shall say that the pair $ (\vr,\bu) $ is a weak solution to the problem   \eqref{i1}--\eqref{i7}  emanating from the initial data $ (\vr_0,\bu_0)$ if:
\begin{description}
\item{(a)}  $ \vr \in C_{\rm weak}([0,T]; L^a(\Omega)), ~\text{for a certain}~ a > 1, \; \vr \ge 0 ~a.e. ~\text{in}~ (0,T)\time\Omega,$  and $ \bu \in L^2(0,T; W_0^{1,2}(\Omega; \R^3)).$

\item{(b)} the continuity equation  $(\ref{i1})$ is satisfied in the following weak sense
\begin{equation}\label{contf}
	\int_{\Omega} \vr \varphi \dx\Big|_0^\tau  = \int_{0}^{\tau}\int_\Omega \Big(\vr  \partial_t \varphi +  \vr \bu \cdot \nabla_x \varphi\Big) \dx\dt, \; \forall \tau \in [0,T], \,  \forall \varphi \in C^{\infty}_c ( [0,T]\times \overline{\Omega}).
\end{equation}
\item{(c)} $\vr\bu\in C_{\rm weak}([0,T]; L^b(\Omega; \R^3))$, for a certain $b > 1$, and the momentum equation $(\ref{i2}) $ is satisfied in the weak sense,\begin{multline}\label{movf}
	\int_{\Omega} \vr\bu \cdot \varphi \dx\Big|_0^\tau = \int_0^\tau \int_\Omega \Big(\vr\bu \cdot \partial_t \varphi+ \vr\bu\otimes\bu:\nabla\varphi
	+ p(\vr) \dv\varphi\Big)\dx\dt \\
	- \int_0^\tau \int_\Omega\Big(\mu\nabla \bu : \nabla_x \varphi \dx\dt +(\mu +\lambda){\rm div}\bu{\rm div}\varphi\Big)\dx\dt, \; \forall \tau \in [0,T], \, \forall  \varphi \in C^{\infty}_c ( [0,T] \times \Omega;\R^3).
\end{multline}
\item {(d)} The following energy inequality is satisfied
\begin{equation}\label{ienergief}
	\int_\Omega \Big(\frac{1}{2} \vr|\bu|^2 + {{H}}(\vr)\Big) \dx\Big|_0^\tau + \int_0^\tau \int_\Omega\Big( \mu| \nabla \bu |^2 +(\mu+\lambda)|\dv\bu|^2\Big) \dx\dt\le 0, \mbox{ for a.a. } \tau\in (0,T),
\end{equation}
\begin{equation}\label{H}
	{ \mbox{ with } \; H(\vr)=\vr\int_1^\vr\frac{p(z)}{z^2}{\rm d}z.}
\end{equation}
\end{description}
Here and hereafter the symbol $\displaystyle \int_{\Omega} g \dx\,|_0^\tau$ is meant for $\displaystyle  \int_\Omega g(\tau,x)\dx - \int_\Omega g_0(x)\dx$.
\end{df}

In the above definition, we tacitly assume that all  the integrals in the formulas \eqref{contf}--\eqref{ienergief} are defined and we recall that
 $C_{\rm weak}([0,T]; L^a(\Omega))$  is the space of functions of $L^\infty([0,T]; L^a(\Omega))$ which are continuous
as functions of time in the weak topology of the space $L^a(\Omega)$.

{ We notice that the function $\vr\mapsto H(\vr)$ is a solution of the ordinary differential equation
$\vr H'(\vr)-H(\vr)=p(\vr)$
with the constant of integration fixed such that $H(1)=0$.}

Note that the existence of weak solutions emanating from the finite energy initial data is well-known on bounded Lipschitz domains  provided
 $\gamma> 3/2$,  see Lions \cite{LI4} for `large' values of $\gamma$, Feireisl and coauthors \cite{FNP} for $\gamma>3/2$.

\bProposition{r1}

Suppose the $\Omega \subset R^3$ is a bounded domain of class $C^3$. Let $r$, $V$ be a  weak solution to problem (\ref{i1}--\ref{i7}) in
$(0,T) \times \Omega$, originating from the initial data
\bFormula{r1}
r_0 \in C^3(\Ov{\Omega}), \ r_0 > 0 \ \mbox{in} \ \Ov{\Omega},
\eF
\bFormula{r2}
\vc{V}_0 \in C^3 (\Ov{\Omega}; R^3),
\eF
satisfying the compatibility conditions
\bFormula{r3}
\vc{V}_0|_{\partial \Omega} = 0, \ \Grad p(r_0)|_{\partial \Omega} = \Div \tn{S}(\Grad \vc{V}_0 )|_{\partial \Omega},
\eF
and such that
\bFormula{r4}
0 \leq r \leq \Ov{r} \ \mbox{a.a. in}\ (0,T) \times \Omega.
\eF

Then $r$, $\vc{V}$ is a classical solution satisfying the bounds:
\bFormula{r5}
\| 1/ r \|_{C([0,T] \times \Ov{\Omega})} + \| r \|_{C^1([0,T] \times \Ov{\Omega})}
+ \| \partial_{t} \Grad r \|_{C([0,T]; L^6(\Omega; R^3))} +
\| \partial^2_{t,t} r \|_{C([0,T]; L^6(\Omega))}
 \leq D ,
\eF
\bFormula{r6}
\| \vc{V} \|_{C^1 ([0,T] \times \Ov{\Omega}; R^3)} + \| \vc{V} \|_{C([0,T]; C^{2}(\Ov{\Omega}; R^3))} +
\| \partial_t \Grad \vc{V} \|_{C([0,T]; L^6(\Omega; R^{3 \times 3}))}   + \| \partial^2_{t,t} \vc{V} \|_{L^2(0,T; L^6(\Omega))}
\leq D,
\eF
where $D$ depends on $\Omega$, $T$, $\Ov{r}$, and the initial data $r_0$, $\vc{V}_0$ (via $\|(r_0,\vc V_0)\|_{C^3(\overline\Omega;R^4)}$ and $\min_{x\in\overline\Omega} r_0(x)$).

\eP
\bProof

The proof will be carried over in several steps.

\medskip

{\bf Step 1}

According to Cho, Choe, and Kim \cite{ChoChoeKim}, problem (\ref{i1}--\ref{i7}) admits a strong solution unique in the class
\bFormula{r7}
r \in C([0,T_M); W^{1,6}(\Omega)), \ \partial_t r \in C([0,T_M); L^6(\Omega)), \;1/r\in L^\infty(Q_T),
\eF
\bFormula{r8}
\vc{V} \in C([0,T_M]; W^{2,2}(\Omega; R^3)) \cap L^2(0, T_M; W^{2,6}(\Omega; R^3)),\
\partial_t \vc{V} \in L^2(0, T_M; W^{1,2}_0 (\Omega; R^3)).
\eF
defined on a time interval $[0, T_M)$, where $T_M > 0$ is finite or infinite and depends on the initial data.
Moreover, for any $T_M^*<T_M$, there is a constant $c=c(T_M^*)$ such that
\begin{equation}\label{estCCK}
\|r\|_{L^\infty(0, T_M^*;W^{1,6}(\Omega))}+\|\partial_t r\|_{L^\infty(0, T_M^*;L^{6}(\Omega))} +\|1/r\|_{L^\infty(Q_T)}
\end{equation}
$$
+ \|\vc V\|_{L^\infty(0, T_M^*;W^{2,2}(\Omega; R^3))}+\|\vc V\|_{L^2(0, T_M^*;W^{2,6}(\Omega; R^3))}+\|\partial_t \vc V\|_{L^2(0, T_M^*;W^{1,2}(\Omega))}
$$
$$
\le c\Big(\|r_0\|_{W^{1,6}(\Omega)}+\|\vc V_0\|_{W^{2,2}(\Omega)}\Big).
$$

\medskip

{\bf Step 2}

By virtue of the weak-strong uniqueness result stated in \cite[Theorem 4.1]{FeJiNo} (see also \cite[Theorem 4.6]{FENOSU}), the weak solution $r$, $\vc{V}$ coincides on the time interval $[0,T_M)$ with the strong solution, the existence of which is claimed in the previous step. According to Sun, Wang, Zhang \cite[Theorem 1.3]{SuWaZha}, if $T_M<\infty$ then
$$
\limsup_{t\to T_M-}\|r(t)\|_{L^\infty(\Omega)}=\infty.
$$
Since (\ref{r4}) holds, we infer that $T_M=T$. At this point we conclude that couple $(r,\vc V)$ possesses regularity (\ref{r7}--\ref{r8}) and that that the bound (\ref{estCCK}) holds with $c$ dependent solely on $T$.

\medskip

{\bf Step 3}

Since the initial data enjoy the regularity and compatibility conditions stated in (\ref{r1}--\ref{r3}), a straightforward bootstrap argument gives rise to better bounds, specifically,
the solution belongs to the Valli-Zajaczkowski (see \cite[Theorem 2.5]{VAZA}) class
\bFormula{r9}
r \in C([0,T]; W^{3,2}(\Omega)), \ \partial_t r  \in L^2(0,T; W^{2,2} (\Omega)),
\eF
\bFormula{r10}
\vc{V} \in C([0,T]; W^{3,2}(\Omega)) \cap L^2(0,T; W^{4,2}(\Omega; R^3)), \ \partial_t \vc{V} \in L^2(0,T; W^{2,2} (\Omega; R^3)),
\eF
where, similarly to the previous step, the norms depend only on the initial data, $\Ov{r}$, and $T$.

\medskip

{\bf Step 4}

We write equation (\ref{i2}) in the form
\bFormula{r10+}
\partial_t \vc{V} - \frac{1}{r} \Div \tn{S}(\Grad \vc{V}) = - \vc{V} \cdot \Grad \vc{V} + \frac{1}{r} \Grad p(r),
\eF
where, by virtue of (\ref{r10}) and a simple interpolation argument, $\vc{V} \in C^{1+ \nu}([0,T] \times \Ov{\Omega}; R^{3 \times 3})$, and, by the same token  $ r \in C^{1+ \nu}([0,T] \times \Ov{\Omega})$ for some $\nu > 0$. Consequently, by means of the standard theory of parabolic equations, see for instance Ladyzhenskaya et al. \cite{LADSOLUR}, we may infer that $r$, $\vc{V}$ is a classical solution,
\bFormula{r11}
\partial_t \vc{V}, \ \Grad^2 \vc{V} \ \mbox{H\" older continuous in}\ [0,T] \times \Ov{\Omega}.
\eF
and, going back to (\ref{i1}),
\bFormula{r12}
\partial_t r \ \mbox{H\" older continuous in}\ [0,T] \times \Ov{\Omega}.
\eF

\medskip

{\bf Step 5}

We write
\[
\Grad \partial_t r = - \Grad \vc{V} \cdot \Grad r - \vc{V} \cdot \Grad^2 r - \Grad r \Div \vc{V} - r \Grad \Div \vc{V};
\]
whence, by virtue (\ref{r9}), (\ref{r11}), (\ref{r12}), and the Sobolev embedding $W^{1,2} \hookrightarrow L^6$,
\bFormula{r13}
\partial_t r \in C([0,T]; W^{1,6}(\Omega)).
\eF

Next, we differentiate (\ref{r10+}) with respect to $t$. Denoting $\vc{Z} = \partial_t \vc{V}$ we therefore obtain
\bFormula{r11+}
\partial_t \vc{Z} - \frac{1}{r} \Div \tn{S}(\Grad \vc{Z}) + \vc{V} \cdot \Grad \vc{Z} = \partial_t \left( \frac{1}{r} \right) \Div \tn{S}(\Grad \vc{V})
- \partial_t \vc{V} \cdot \Grad \vc{V}  + \partial_t \left( \frac{1}{r} \Grad p(r) \right),
\eF
where, in view of (\ref{r13}) and the previously established estimates, the expression on the right-hand side is bounded in
$C([0,T]; L^6(\Omega; R^3))$. Thus using the $L^p-$maximal regularity (see Denk, Hieber, and Pruess \cite{DEHIEPR}, Krylov \cite{Krylov} or   Danchin \cite[Theorem 2.2]{DANCHIN} ), we deduce that
\bFormula{r14}
\partial^2_{t,t} \vc{V} = \partial_t \vc{Z} \in L^2(0,T; L^6(\Omega; R^3)),\
\partial_t \vc{V} = \vc{Z} \in C([0,T]; W^{1,6} (\Omega; R^3)).
\eF

Finally, writing
\[
\partial^2_{t,t} r = - \partial_t \vc{V} \cdot \Grad r - \vc{V} \cdot \partial_t \Grad r - \partial_t r \Div \vc{V} -
r \partial_t \Div \vc{V},
\]
and using (\ref{r13}), (\ref{r14}), we obtain the desired conclusion
\[
\partial^2_{t,t} r \in C([0,T]; L^6(\Omega)).
\]

\qed

Here and hereafter, we shall use notation $a\aleq b$ and $a\approx b$. the symbol $a\aleq b$ means that there exists
$c=c(\Omega, T,\mu,\gamma)>0$ such that $a\le c b$; $a\approx b$ means $a\aleq b$ and $b\aleq a$.

\subsection{Extension lemma}

\bLemma{Ex1}
Under the hypotheses of Proposition \ref{Pr1}, the functions $r$ and $\vc{V}$ can be extended outside $\Omega$ in such a way that:
\begin{description}
\item {(1)}
The extended functions (still denoted by $r$ and $\vc{V}$) are such that ${\vc V}$ is compactly supported in $[0,T]\times \R^3$ and $r\ge\underline r>0$.
\item {(2)}
\bFormula{ex1}
\| \vc{V} \|_{C^1 ([0,T] \times R^3; R^3)} + \| \vc{V} \|_{C([0,T]; C^{2}(R^3; R^3))} +
\| \partial_t \Grad \vc{V} \|_{C([0,T]; L^6(R^3; R^{3 \times 3}))}   + \| \partial^2_{t,t} \vc{V} \|_{L^2(0,T; L^6(R^3))}
\eF
\[
\aleq
\| \vc{V} \|_{C^1 ([0,T] \times \Ov{\Omega}; R^3)} + \| \vc{V} \|_{C([0,T]; C^{2}(\Ov{\Omega}; R^3))} +
\| \partial_t \Grad \vc{V} \|_{C([0,T]; L^6(\Omega; R^{3 \times 3}))}   + \| \partial^2_{t,t} \vc{V} \|_{L^2(0,T; L^6(\Omega))};
\]
\item {(3)}
\bFormula{ex2}
\| r \|_{C^1([0,T] \times R^3)}
+ \| \partial_{t} \Grad r \|_{C([0,T]; L^6(R^3; R^3))} +
\| \partial^2_{t,t} r \|_{C([0,T]; L^6(R^3))}
\eF
\[
\aleq \| r \|_{C^1([0,T] \times \Ov{\Omega})}
+ \| \partial_{t} \Grad r \|_{C([0,T]; L^6(\Omega; R^3))} +
\| \partial^2_{t,t} r \|_{C([0,T]; L^6(\Omega))} +
\]
\[
\| \vc{V} \|_{C^1 ([0,T] \times \Ov{\Omega}; R^3)} + \| \vc{V} \|_{C([0,T]; C^{2}(\Ov{\Omega}; R^3))} +
\| \partial_t \Grad \vc{V} \|_{C([0,T]; L^6(\Omega; R^{3 \times 3}))}   + \| \partial^2_{t,t} \vc{V} \|_{L^2(0,T; L^6(\Omega))};
\]
\item {(4)}
\bFormula{ex3}
\partial_t r + \Div (r \vc{V}) = 0 \ \mbox{in}\ (0,T) \times R^3.
\eF

\end{description}
\eL
\bProof  { We first construct the extension of the vector field $\vc V$. To this end, we follow  the standard construction in the flat domain, see Adams \cite[Chapter 5, Theorem 5.22]{A} and combine it with the standard
procedure of `flattening' of the boundary and the partition of unity technique, we get (\ref{ex1}) Once this is done, we solve on the whole space the transport equation (\ref{ex3}). It is easy to show that the unique solution $r$ of this equation possesses regularity and estimates stated in (\ref{ex2}).}
\qed
\bRemark{r1}
Here and hereafter, we denote
$X_T(\R^3)$ a subset of $L^2((0,T)\times\R^3)$ of couples $(r,\vc V)$, $r> 0$ with finite norm
\begin{equation}\label{XT}
\|(r,\vc V)\|_{X_T(\R^3)}\equiv \| r \|_{C^1([0,T] \times R^3)}
+ \| \partial_{t} \Grad r \|_{C([0,T]; L^6(R^3; R^3))} +
\| \partial^2_{t,t} r \|_{C([0,T]; L^6(R^3))}
\end{equation}
$$
\| \vc{V} \|_{C^1 ([0,T] \times R^3; R^3)} + \| \vc{V} \|_{C([0,T]; C^{2}(R^3; R^3))} +
\| \partial_t \Grad \vc{V} \|_{C([0,T]; L^6(R^3; R^{3 \times 3}))}   + \| \partial^2_{t,t} \vc{V} \|_{L^2(0,T; L^6(R^3))}
$$
We notice that the first component of the couple belonging to $X_T(\R^3)$ is always strictly positive on $[0,T]\times R^3$ and set
\begin{equation}\label{uor}
0<\underline r={\rm min}_{(t,x)\in [0,T]\times \Rm^3}r(t,x),\quad \overline r={\rm max}_{(t,x)\in [0,T]\times \Rm^3}r(t,x)<\infty
\end{equation}
\eR

\subsection{Physical domain, mesh approximation}

The physical space is represented by a bounded domain $\Omega \subset R^3$ of class $C^3$. The numerical domains $\Omega_h$ are polyhedral domains,
\bFormula{mesh}
\Ov{\Omega}_h = \cup_{K \in \mathcal{T}} {K},
\eF
where ${\cal T}$ is a set of tetrahedra which have the following property:
If $K \cap L \ne \emptyset$, $K \ne L$, then $K\cap L$ is either a common face, or a common edge, or a common vertex.
 By ${\cal{E}}(K)$, we denote the set of the  faces  $\sigma$ of the element $K \in {\cal{T}}$.
The set of all faces of the mesh is denoted by ${\cal{E}}$; the set of faces included in the boundary $\partial\Omega_h$ of $\Omega_h$ is denoted by ${\cal{E}}_{\extt}$ and the set of internal faces (i.e ${\cal{E}} \setminus {\cal{E}}_{\extt} $) is denoted by ${\cal{E}}_{\intt}$.

Further, we ask
\bFormula{vertex}
\mathcal{V}_h \in \partial \Omega_h \ \mbox{a vertex} \ \Rightarrow \ \mathcal{V}_h \in \partial \Omega.
\eF
Furthermore, we suppose that each $K$ is a tetrahedron such that
\begin{equation}\label{reg1}
\xi [K] \approx {\rm diam}[K] \approx h,
\end{equation}
where $\xi[K]$ is the radius of the largest ball contained in $K$.

The properties of this mesh needed in the sequel are formulated in the following lemma, whose (easy) proof is left to the reader.

\bLemma{N1}

There exists a positive constant $d_\Omega$ depending solely on the geometric properties of $\partial \Omega$ such that
\[
{\rm dist}[x, \partial \Omega] \leq d_\Omega h^2,
\]
for any $x \in \partial \Omega_h$. Moreover,
$$
|(\Omega_h\setminus\Omega)\cup (\Omega\setminus\Omega_h)|\aleq h^2.
$$

\eL

We find important to emphasize that $\Omega_h \not\subset \Omega$, in general.

\subsection{Numerical spaces}
We denote by $Q_h(\Omega_h)$ the space of piecewise constant functions:
\begin{equation}\label{Q}
Q_h(\Omega_h)=\{q\in L^2(\Omega_h)\,|\, ~\forall K \in {\cal{T}}, ~q_{|K}\in \R\}.
\end{equation}

For a function $v$ in $C(\overline\Omega_h)$, we set
\begin{equation}
 	v_K = \frac 1{|K|}\int_K v \dx \textrm{ for } K \in {\cal T} \textrm{  and }   \Pi^Q_h v{ (x)}= \sum_{K\in {\cal T}} v_K 1_K{ (x)},\; { x\in \Omega}.
	\label{vhat}
\end{equation}
{ Here and in what follows, $1_K$ is the characteristic function of $K$}.

We define the Crouzeix-Raviart space with `zero traces':
\begin{equation}\label{CR0}
	 V_{h,0}(\Omega_h) = \{ v \in L^2(\Omega_h), ~\forall K \in {\cal{T}}, ~v_{|K} \in \mathbb{P}_1(K),
\end{equation}
$$
	 \forall \sigma \in {\cal{E}}_{\intt} ,\; \sigma=K|L,\; \int_{\sigma}v_{|K} \dS=\int_{\sigma}v_{|L}\dS,\quad
	\forall \sigma' \in {\cal{E}}_{\extt},\; \int_\sigma' v \dS=0 \},
$$
and `with general traces'
\begin{equation}\label{CR}
	 V_{h}(\Omega_h) = \{ v \in L^2(\Omega), ~\forall K \in {\cal{T}}, ~v_{|K} \in \mathbb{P}_1(K),\;
\forall \sigma \in {\cal{E}}_{\intt} ,\; \sigma=K|L,\; \int_{\sigma}v_{|K} \dS=\int_{\sigma}v_{|L}\dS\}.
\end{equation}
We denote by  $\Pi^V_h$ the standard Crouzeix-Raviart projection, and $\Pi^V_{h,0}$ the Crouzeix-Raviart projection with
`zero trace', specifically,
\[
\Pi^V_h: C(\overline\Omega_h)\rightarrow V_h(\Omega_h),\;\int_\sigma \Pi^V_h[ \phi] \ {\rm dS}_x = \int_\sigma \phi \ {\rm dS}_x \ \mbox{for all}\ \sigma \in \mathcal{E},
\]
\[
\Pi^V_{h,0}: C(\overline\Omega_h)\rightarrow V_h(\Omega_h), \;
\int_\sigma \Pi^V_{h,0}[ \phi] \ {\rm dS}_x = \int_\sigma \phi \ {\rm dS}_x \ \mbox{for all}\ \sigma \in \mathcal{E}_{\rm int},
\;\; \int_{\sigma} \Pi^V_{h,0}[ \phi] \ {\rm dS}_x = 0 \ \mbox{whenever}\ \sigma \in \mathcal{E}_{\rm ext} .
\]

{ If}  $v \in W^{1,1}(\Omega_h)$,  we set
\begin{equation}
	v_\sigma=\frac 1{|\sigma|}\int_\sigma v{\rm d} S\textrm{ for } \sigma\in {\cal E}.
\label{vtilde}
\end{equation}

Each element $v\in V_h(\Omega_h)$ can be written in the form
\begin{equation}\label{CRP}
	v{ (x)}=\sum_{\sigma\in{\cal E}}v_\sigma\varphi_\sigma{ (x)},\quad  x\in \Omega_h,
\end{equation}
where the set $\{\varphi_\sigma\}_ {\sigma\in{\cal E}}\subset V_h(\Omega_h)$ is the classical Crouzeix-Raviart basis determined by
\begin{equation}\label{CRB}
\forall (\sigma,\sigma')\in {\cal E}^2,\;\frac 1{|\sigma'|}\int_{\sigma'}\varphi_\sigma \dS=\delta_{\sigma,\sigma'}.
\end{equation}

Similarly, each element $v\in V_{h,0}(\Omega_h)$ can be written in the form
\begin{equation}\label{CRP+}
	v{ (x)}=\sum_{\sigma\in{\cal E}_{\rm int}}v_\sigma\varphi_\sigma{ (x)},\quad  x\in \Omega_h.
\end{equation}


We first recall in Lemmas \ref{Lemma1}--\ref{Lemma6} the standard properties of the projection $\Pi^V_h$. The collection of their proofs in the requested generality can be found in the Appendix of \cite{GHMN} with exception of Lemma \ref{LN2} and its Corollary \ref{CN1}. We refer to the monograph of Brezzi, Fortin \cite{BRFO}, the Crouzeix's and Raviart's paper \cite{cro-73-con}, Gallouet, Herbin, Latch\'e \cite{GHL2009iso} for the original versions of some of these proofs. We present the proof of Lemma \ref{LN2}
dealing with the comparison of projections $\Pi^V_h$ and $\Pi^V_{h,0}$ that we did not find in the literature.

\begin{Lemma} \label{Lemma1}
The following estimates hold true:
\begin{equation}\label{ddd}
\|\Pi^V_h[\phi]\|_{L^\infty(K)}+ \|\Pi^V_{h,0}[\phi]\|_{L^\infty(K)}\aleq\|\phi\|_{L^\infty(K)},
\end{equation}
for all $K\in {\cal T}$ and $\phi\in C(K)$;

\begin{equation}\label{L1-2}
 \|\phi-\Pi^V_h[\phi]||_{L^p(K)} \aleq h^{s}\|\nabla^s \phi\|_{L^p(K;\Rm^{d^s})},\; s=1,2,\; 1\le p\le\infty,
\end{equation}
and
\begin{equation}\label{L1-3}
|| \nabla( \phi- \Pi^V_h[\phi]) ||_{L^p(K;\Rm^d)} \le c h^{s-1} \|\nabla^s \phi\|_{L^p(K;\Rm^{d^s})},\; s=1,2,\; 1\le p\le\infty,
\end{equation}
for all $K\in {\cal T}$ and $\phi\in C^s(K)$.
\end{Lemma}

\begin{Lemma}\label{Lemma4}
Let $1\le p<\infty$. Then
\begin{equation}\label{norms1}
\sum_{\sigma\in {\cal E}} |\sigma| h |v_\sigma|^p\approx||v||^p_{L^p(\Omega_h)},
\end{equation}
with any $v\in V_h(\Omega_h)$.
\end{Lemma}

\begin{Lemma}\label{Lemma2+}
The following Sobolev-type inequality holds true:
\begin{equation}\label{sob1}
||v||^2_{L^{6}(\Omega_h)} \aleq \sum_{K\in {\cal T}}\int_K|\Grad v|^2{\rm d} x,
\end{equation}
with any $v\in V_{h,0}(\Omega_h)$.
\end{Lemma}

\begin{Lemma}
There holds:
\begin{equation}\label{L1-1}
\sum_{K\in {\cal T}}\int_K q ~\dv\Pi_h^V[\bv] \dx = \int_{\Omega} q ~\dv \bv \dx,
\end{equation}
for all $\bv \in C^1(\overline\Omega_h,\R^d)$ and all $q \in Q_h(\Omega_h)$.
\end{Lemma}

\begin{Lemma}[Jumps over faces in the Crouzeix-Raviart space]\label{Lemma6}
For all $ v \in V_{h,0}(\Omega_h)$ there holds
\begin{equation}\label{tbound}
\sum_{\sigma \in {\cal{E}}} \frac{1}{h} \int_\sigma [v]_{\sigma,\bn_\sigma}^2 \dS \aleq \sum_{K\in {\cal T}}\int_K|\Grad v|^2{\rm d} x,
\end{equation}
where  $[v]_{\sigma,\bn_\sigma}$ is a jump of $v$  with respect to a normal $\bn_\sigma$ to the face $\sigma$,
\[
\forall x\in\sigma=K|L\in{\cal E}_{\rm int},\quad [v]_{\sigma,\bn_\sigma}(x)=\left\{
\begin{array}{c}
 v|_K(x)-v|_L(x)\;\mbox{if $\bn_\sigma=\bn_{\sigma,K}$}\\
v|_L(x)-v|_K(x)\;\mbox{if $\bn_\sigma=\bn_{\sigma,L}$}
\end{array}\right. ,
\]
($\bn_{\sigma,K}$ is the normal of $\sigma$, that is outer w.r. to element $K$) and
\[
\forall x\in\sigma\in{\cal E}_{\rm ext},\quad [v]_{\sigma,\bn_\sigma}(x)=v(x),\;\mbox{with $\bn_\sigma$ an exterior normal to $\partial\Omega$}.
\]
\end{Lemma}

We will need to compare the projections $\Pi_h^V$ and $\Pi_{h,0}^V$. Clearly they coincide on `interior' elements
meaning $K\in {\cal T}$, $K\cap\partial\Omega_h=\emptyset$. We have the following lemma for the tetrahedra with non void intersection with the boundary.

\bLemma{N2}
We have
\begin{equation}\label{N2-1}
\| \Pi^V_h [\phi] - \Pi^V_{h,0}[\phi] \|_{L^\infty(K)} + h \| \Grad (\Pi^V_h [\phi] - \Pi^V_{h,0}[\phi] ) \|_{L^\infty(K;R^3)}
\aleq
\sup_{\sigma \subset K\cap\partial_{\Omega_h} } \| \phi \|_{L^\infty(\sigma) }
\;\mbox{if $K\in {\cal T}$, $K\cap\partial_{\Omega_h}\neq\emptyset$},
\end{equation}
for any $\phi \in C(K)$.
\eL
\bProof
We recall the Crouzeix-Raviart basis (\ref{CRB}) and the fact that $\Pi^V_h$ and $\Pi^V_{h,0}$ differ only in basis functions corresponding to $\sigma \in \mathcal{E}_{\rm ext}$. We have

\begin{equation}\label{projections_difference}
\| \Pi^V_h [\phi]- \Pi^v_{h,0}[\phi] \|_{L^\infty(K)} \leq \left\| \sum_{\sigma \in \mathcal{E}(K) \cap \mathcal{E}_{\rm ext} } \varphi_\sigma \frac{1}{|\sigma|} \int_\sigma \phi \dS \right\|_{L^\infty(K)} \leq c(K) \cdot \sup_{\sigma \in \mathcal{E}(K) \cap \mathcal{E}_{\rm ext}} \| \phi \|_{L^\infty(\sigma)},
\end{equation}
and
\begin{equation*}
\begin{split}
& h \| \nabla_x (\Pi^V_h [\phi] - \Pi^V_{h,0} [\phi])\|_{L^\infty(K)} \leq h \left\| \sum_{\sigma \in \mathcal{E}(K) \cap \mathcal{E}_{\rm ext} } \nabla_x \varphi_\sigma \frac{1}{|\sigma|} \int_\sigma \phi \dS \right\|_{L^\infty(K)} \\ & \qquad \, \qquad \leq ch \sup_{\sigma \subseteq K \cap \partial \Omega_h} \| \phi \|_{L^\infty(\sigma)} \left\| \sum_{\sigma \in \mathcal{E}(K) \cap \mathcal{E}_{\rm ext} } \nabla_x \varphi_\sigma \right\|_{L^\infty(K)}.
\end{split}
\end{equation*}
The proof is completed by $ \| \sum_{\sigma \in \mathcal{E}(K) \cap \mathcal{E}_{\rm ext} } \nabla_x \varphi_\sigma \|_{L^\infty(K)} \leq c(K)h^{-1}$.
\qed

In fact, in the derivation of the error estimates we will use the  consequence of the above observations formulated in the following two corollaries.

\bCorollary{N1}
Let $\phi \in C^1(R^3)$ such that $\phi|_{\partial \Omega} = 0$.
Then we have,
\begin{equation}\label{N2-2}
\| \Pi^V_h [\phi] - \Pi^V_{h,0}[\phi] \|_{L^\infty(K)}=0\;\mbox{if $K\in {\cal T}$, $K\cap\partial_{\Omega_h}=\emptyset$,}
\end{equation}
\begin{equation}\label{N2-3}
\| \Pi^V_h [\phi] - \Pi^V_{h,0}[\phi] \|_{L^\infty(K)} + h \| \Grad (\Pi^V_h [\phi] - \Pi^V_{h,0}[\phi] ) \|_{L^\infty(K;R^3)}
\aleq h^2  \| \Grad \phi \|_{L^\infty(R^3;R^3) },
\end{equation}
if $K \in \mathcal{T}_h$, $K \cap \partial \Omega_h \ne \emptyset$, $\partial K \not \subset \partial \Omega$.
\eC

\bProof
Relation (\ref{N2-2}) follows immediately from (\ref{N2-1}), as there is an empty sum on the right hand side for `interior' elements ($K \cap \partial \Omega_h = \emptyset$).

For any $x \in \partial \Omega_h$ there exists $y \in \partial \Omega$ (and thus $\phi(y)=0$) such that

\begin{equation}\label{est_phisigma}
|\phi(x)| \leq {\rm dist}[x,y] \| \nabla_x \phi\|_{L^\infty(R^3;R^3)} \aleq h^2\| \nabla_x \phi\|_{L^\infty(R^3;R^3)},
\end{equation}
where we used Lemma \ref{LN1} for the latter inequality. The proof is completed by taking supremum over $K \in \mesh_h$ and combining with (\ref{est_phisigma}). Note that the mesh regularity property (\ref{reg1}) supplies a uniform estimate of constants $c(K)$ from the previous lemma, which enables to write the latter inequality in (\ref{est_phisigma}).
\qed

\bCorollary{N2}
For any $\phi\in C(R^3)$,
\begin{equation}\label{N2-4}
\| \Pi^V_h [\phi] - \Pi^V_{h,0}[\phi] \|_{L^p(K)}\aleq h^{3/p}\|\phi\|_{L^\infty(\overline\Omega_h)},\;1\le p<\infty.
\end{equation}
\eC

\bProof
Apply inverse estimates (see e.g. \cite[Lemma~2.9]{KARPER}) to (\ref{N2-1}).
\qed
%

We will frequently use the Poincar\'e, Sobolev  and interpolation inequalities on tetrahedra reported
in the following lemma.

\begin{Lemma}\label{Lemma2}
\medskip\noindent
\begin{description}
\item{\it (1) } We have,
\begin{equation}\label{L2-1}
\|v-v_K\|_{L^p(K)}\aleq h\|\nabla v\|_{L^p(K)},
\end{equation}
\begin{equation}\label{L2-2}
\forall \sigma\in {\cal E}(K), \;\|v-v_\sigma\|_{L^p(K)}\aleq  h\|\nabla v\|_{L^p(K)},
\end{equation}
for any $v\in W^{1,p}(K)$, where $1\le p\le\infty$.
\item{\it (2) } There holds
\begin{equation}\label{L2-3}
\|v-v_K\|_{L^{p^*}(K)} \aleq \|\nabla v\|_{L^p(K)},
\end{equation}
\begin{equation}\label{L2-4}
\forall \sigma\in {\cal E(}K), \;\|v-v_\sigma\|_{L^{p^*}(K)}\aleq \|\nabla v\|_{L^p(K)},
\end{equation}
 for any $v\in W^{1,p}(K)$, $1\le p<d$, where $p^*=\frac{dp}{d-p}$.
\item{\it (3) } We have,
\begin{equation}\label{interpol1}
\|v-v_K\|_{L^{q}(K)}
\le c  h^\beta\|\nabla v\|_{L^p(K;\Rm^d)},
\end{equation}
\begin{equation}\label{interpol2}
\|v-v_\sigma\|_{L^{q}(K)}
\le ch^\beta \|\nabla v\|_{L^p(K;\Rm^d)},
\end{equation}
for any $v\in W^{1,p}(K)$, $1\le p<d$,
where $\frac 1q=\frac \beta p+\frac {1-\beta}{p^*}$.
%
\end{description}
\end{Lemma}

We finish the section of preliminaries by recalling two algebraic inequalities 1) the `imbedding' inequality
\begin{equation}\label{dod1*}
	\Big(\sum_{i=1}^L|a_i|^p\Big)^{1/p}\le \Big(\sum_{i=1}^L|a_i|^q\Big)^{1/q},
\end{equation}
for all $a=(a_1,\ldots,a_L)\in \mathbb{R}^L$, $1\le q\le p<\infty$ and the discrete H\"older inequality
\begin{equation}\label{dod2*}
	\sum_{i=1}^L|a_i||b_i|\le \Big(\sum_{i=1}^L|a_i|^q\Big)^{1/q} \Big(\sum_{i=1}^L|a_i|^p\Big)^{1/p},
\end{equation}
for all $a=(a_1,\ldots,a_L)\in \mathbb{R}^L$, $b=(b_1,\ldots,b_L)\in \Rm^L$, $\frac 1 q+\frac 1p=1$.

\section{Main result}

Here and hereafter we systematically use the following abbreviated  notation:
\begin{equation}\label{n1}
\hat \phi=\Pi^Q_h[\phi],\;\;\phi_h= \Pi^V_h [\phi],\;\; \phi_{h,0}= \Pi^V_{h,0} [\phi].
\end{equation}
For a function $v\in C([0,T], L^1(\Omega))$ we set
\begin{equation}\label{n2}
v^n(x)=v(t_n,x),
\end{equation}
where $t_0=0<t_1<\ldots<t_{n-1}<t_n<t_{n+1}<\ldots t_N=T$ is a partition of the interval $[0,T]$.
Finally, for a function $v\in V_h(\Omega_h)$ we denote
\begin{equation}\label{n3}
\nabla_h v(x)=\sum_{K\in{\cal T}} \Grad v(x)1_{K}(x),\;\; {\rm div}_h{\vc v}(x)= \sum_{K\in{\cal T}} {\rm div}_x\vc v(x)1_{K}(x).
\end{equation}

In order to ensure the positivity of the approximate densities, we shall use an upwinding technique for the density in the mass equation.
For  $q\in Q_h(\Omega_h)$ and  $\bu\in \vc V_{h,0} (\Omega_h;\R^3)$,
the upwinding of  $q$ with respect to  $\bu$ is defined, for $\sigma =K|L\in {\cal E}_{\rm{int}}$ by
\begin{equation}\label{upwind1}
q^{{\rm up}}_\sigma=\begin{cases}
                      q_K \;\mbox{if }\bu_\sigma\cdot\bn_{\sigma, K}> 0\\ q_L\; \mbox{if } \bu_\sigma\cdot\bn_{\sigma, K}{ \le} 0
                    \end{cases},
\end{equation}
and we denote
\[ {\rm Up}_K(q,\vc u)\equiv
\sum_{\sigma\in{\cal E}(K)c\cap {\cal E}_{\rm int}} q_\sigma^{\rm up}{\bu}_\sigma\cdot{\vc n}_{\sigma,K}= \sum_{\sigma\in{\cal E}(K)\cap {\cal E}_{\rm int}} \Big(q_K [{\bu}_\sigma\cdot{\vc n}_{\sigma,K}]^+ + q_L [{\bu_\sigma}\cdot{\vc n}_{\sigma,K}]^-\Big),
\]
where $a^+ = \max(a,0)$, $a^- = \min(a,0)$.

\subsection{Numerical scheme}

{\it We consider a couple $(\vr^n,\bu^n)=(\vr^{n,(\deltat,h)},\bu^{n,(\deltat,h)})$ of (numerical) solutions of the following algebraic system (numerical scheme):
\bFormula{num1}
\vr^n \in Q_h (\Omega_h), \ \vr^n > 0, \ \vu^n \in V_{h,0} (\Omega_h; R^3),\quad n=0,1,\ldots,N,
\eF

\bFormula{num2}
\sum_{K \in \mathcal{T} } |K| \frac{ \vr^n_K - \vr^{n-1}_K }{ \Det } \phi_K + \sum_{K \in \mathcal{T} }
\sum_{ \sigma \in \mathcal{E}(K) } |\sigma| \vr^{n,{\rm up}}_\sigma ( \vu^n_{\sigma} \cdot \vc{n}_{\sigma, K} ) \phi_K = 0
\ \mbox{for any}\ \phi \in Q_h(\Omega_h)\ \mbox{and}\ n=1,\ldots,N,
\eF
\bFormula{num3}
\sum_{K\in{\cal T}}\frac{|K|}{\deltat} \Big({\vr^n_K{{\hat \bu}}^n_{K}- \vr^{n-1}_K{{\hat \bu}}^{n-1}_{K}} \Big)\cdot \bv_K+ \sum_{K\in {\cal T}}\sum_{\sigma \in {\cal E}(K)} |\sigma|\vr^{n,{\rm up}}_\sigma {{ \hat\bu}}_{\sigma}^{n,{\rm up}}[\bu^n_\sigma\cdot \bn_{\sigma,K}]\cdot \bv_K
\eF
\[
- \sum_{K\in{\cal T}}p(\vr^n_K)\sum_{\sigma \in {\cal E}(K)}  |\sigma|\bv_\sigma\cdot {\vc n}_{\sigma,K}+\mu\sum_{K\in{\cal T}}\int_K \nabla\bu^n : \nabla\bv \  \dx
\]
\[
+ \frac{\mu}{3} \sum_{K\in{\cal T}}\int_K{\rm div}\bu^n{\rm div}\bv\,  \dx =0,
 \ \mbox{for any}\ \bv \in {V_{h,0}(\Omega;R^3) }\ \mbox{and}\ n=1,\ldots,N.
\]
}

The numerical solutions depend  on the size h of the space discretisation and on the time step $\deltat$.
For the sake of clarity and in order to simplify notation we will always systematically write in all formulas $(\vr^n,\bu^n)$ instead
of $(\vr^{n,(\deltat,h)},\bu^{n,(\deltat,h)})$.

Existence of a solution to problem (\ref{num1}--\ref{num3}) is well known together with the fact that
any solution $(\vr^n)_{1\le n \le N}\subset (Q_h(\Omega))^N$ satisfies $\vr^n >0$ provided
$\vr^0>0$  thanks to the upwind choice in (\ref{num2}) { (see e.g. \cite{GGHL2008baro,KARPER})}.

\bRemark{upwind}
Throughout the paper, $q_\sigma^{\rm up}$ is defined in (\ref{upwind1}), where $\vu$ is the numerical solution constructed in (\ref{num1}--\ref{num3}).
\eR

\subsection{Error estimates}

The main result of this paper is announced in the following theorem:

\bTheorem{M1}

Let $\Omega \subset R^3$ be a bounded domain of class $C^3$ and let the pressure satisfy (\ref{i4}) with $\gamma\ge 3/2$. Let $\{ \vr^n, \vu^n \}_{0 \leq n \leq N}$ be a family of numerical solutions resulting from the
scheme (\ref{num1}--\ref{num3}). Moreover, suppose there are initial data $[r_0, \vc{V}_0]$ belonging to the regularity class specified in
Proposition \ref{Pr1} and giving rise to a weak solution $[r, \vc{V}]$ to the
initial-boundary value problem (\ref{i1}--\ref{i7}) in $(0,T) \times \Omega$ satisfying
\[
0 \leq r(t,x) \leq \Ov{r} \ \mbox{a.a. in} \ (0,T) \times \Omega.
\]

Then $[r, \vc{V}]$ is regular and there exists a positive number
$$
C=
C\Big( M_0,E_0,\underline r, \overline r, |p'|_{{C^1}[\underline r,\overline r]}, \|(\partial_t r, \nabla r, \bV,
\partial_t\bV, \nabla\bV, \nabla^2\bV)\|_{L^\infty(Q_T;\Rm^{45})},
$$
$$
 \|\partial_t^2r\|_{L^1(0,T;L^{\gamma'}(\Omega))}, \|\partial_t\nabla r\|_{L^2(0,T;L^{6\gamma/{5\gamma-6}}(\Omega;\Rm^3))}, \|\partial_t^2\bV,\partial_t\nabla\bV\|_{L^2(0,T;L^{6/5}(\Omega;\Rm^{12}))},
\Big)
$$
such that
\bFormula{M1}
\sup_{1 \leq n \leq N} \int_{\Omega \cap \Omega_h} \left[ \frac{1}{2} \vr^n | \hat\vu^n - \vc{V}(t_n, \cdot) |^2 +
H(\vr^n) - H'(r(t_n, \cdot))(\vr^n - r(t_n, \cdot)) - H(r(t_n)) \right] \ \dx
\eF
\[
+ \Delta t \sum_{1 \leq n \leq N} \int_{\Omega \cap \Omega_h} \left| \nabla_h \vu^n - \Grad \vc{V} (t_n, \cdot) \right|^2 \ \dx
\]

\[
\leq C \left(\sqrt{\Delta t} + {h}^a +  \int_{\Omega \cap \Omega_h} \left[ \frac{1}{2} \vr^0 | \hat\vu^0 - \vc{V}_0 |^2 +
H(\vr^0) - H'(r_0)(\vr^0 - r_0) - H(r_0)) \right] \ \dx \right),
\]

where
\bFormula{M2}
a = \frac{2 \gamma - 3}{\gamma } \ \mbox{if}\ \frac{3}{2} \leq \gamma \leq 2, \ a = \frac{1}{2} \ \mbox{otherwise}.
\eF

\eT

Note that for $\gamma = 3/2$ Theorem \ref{TM1} gives only uniform bounds on the difference of exact and numerical solution, not the convergence.
\section{Uniform estimates}
If we take $\phi=1$ in formula (\ref{num2}) we get immediately the conservation of mass:
\begin{equation}\label{masscons}
	\forall n=1,...N,\quad \int_{\Omega_h} \vr^n \dx = \int_{\Omega_h} \vr^0 \dx.
\end{equation}

Next Lemma reports the standard energy estimates for the numerical scheme (\ref{num1}--\ref{num3}), see again \cite{GGHL2008baro,KARPER}.
\begin{Lemma}\label{Lestimates}
  Let  $(\vr^n,\bu^n)$ be a solution of the discrete problem (\ref{num1}--\ref{num3}) with the pressure $p$ satisfying (\ref{i4}).
Then there exist
\begin{align*}
&\overline \vr^n_{\sigma}\in [\min(\vr^n_K,\vr^n_L), \max(\vr^n_K,\vr^n_L)],\;\sigma=K|L\in {\cal E}_{\rm int},\; n=1,\ldots,N ,\\
&\overline\vr_K^{n-1,n}\in [\min(\vr^{n-1}_K,\vr^n_K), \max(\vr^{n-1}_K,\vr^n_K)],\; K\in {\cal T},\; n=1,\ldots,N,
\end{align*}
such that
\begin{multline}
\sum_{K\in {\cal T}}{|K|}\Big(\frac 12\vr^m_K|{\bu}^m_K|^2
 +H(\vr_K^m)\Big)
-\sum_{K\in {\cal T}}|K|\Big(\frac 12\vr^{0}_K|{\bu}^{0}_K|^2
+H(\vr_K^{0})\Big)
\\
+\deltat \sum_{n=1}^m\sum_{K\in {\cal T}}\Big(\mu\int_K|\Grad\bu^n|^2 \dx+( \mu +\lambda)\int_K|{\rm div}\bu^n|^2 \dx\Big)
\\
+ [D^{m,|\Delta\bu|}_{\rm time}]+ [D^{m,|\Delta\vr|}_{\rm time}]+ [D^{m,|\Delta\bu|}_{\rm space}] +  [D^{m, |\Delta\vr|}_{\rm space}]= 0,
\label{denergyinequality}
\end{multline}
for all $m=1,\ldots,N$,
where
\begin{subequations}\label{upwinddissipation}
	\begin{align}
		& [D^{m,|\Delta\bu|}_{\rm time}]=\sum_{n=1}^m\sum_{K\in {\cal T}}{|K|}\vr_K^{n-1}\frac {|{\bu}_K^n-{\bu}_K^{n-1}|^2} 2,
		\label{upwinddissipation_1}\\
		&[D^{m,|\Delta\vc \vr|}_{\rm time}]=\sum_{n=1}^m\sum_{K\in {\cal T}}|K|H''(\overline\vr_K^{n-1,n})\frac {|{\vr}_K^n-{\vr }_K^{n-1}|^2} 2,
		\label{upwinddissipation_2}\\
		&[D^{m,|\Delta\bu|}_{\rm space}]=\deltat\sum_{n=1}^m\sum_{\sigma=K|L\in{\cal E}_{\rm int}}|\sigma|\vr_\sigma^{n,{\rm up}}\frac{({\bu}^n_K-{{\bu}}^n_L)^2} 2\;|{\bu}^n_\sigma\cdot\bn_{\sigma,K}|,
		\label{upwinddissipation_3}\\
		&[D^{m, |\Delta\vr|}_{\rm space}]=\deltat\sum_{n=1}^m\sum_{\sigma=K|L \in{\cal E}_{\rm int}}|\sigma|H''(\overline \vr^n_{\sigma})\frac{(\vr^n_K-\vr^n_L)^2}2 \;|{\bu}^n_\sigma\cdot{\vc n}_{\sigma,K}|.
		\label{upwinddissipation_4}
\end{align}
 \end{subequations}
\end{Lemma}

We have the following corollary of Lemma \ref{Lestimates} (see \cite[Lemma 4.1, Lemma 4.2]{GHMN}):
\begin{cor}\label{Corollary1}
Under assumptions of Lemma \ref{Lestimates}, we have:
\begin{description}
\item{(1)}
 There exists $c=c(M_0,E_0)>0$ (independent of $n$, $h$ and $\deltat$) such that
\begin{equation}\label{est0}
k\sum_{n=1}^N\int_K|\nabla_x \bu^n|^2\dx\le c,
\end{equation}
\begin{equation}\label{est1}
k\sum_{n=1}^N \|\bu^n\|^2_{L^6(\Omega_h;\Rm^3)}\le c,
\end{equation}
\begin{equation}\label{est2}
{\rm sup}_{n=0,\ldots N}\|\vr^n\hat{\bu^n}^2\|_{L^1(\Omega_h)}\le c.
\end{equation}
 \item{(2)}
\begin{equation}\label{est3}
{\rm sup}_{n=0,\ldots N}\|\vr^n\|_{L^\gamma(\Omega_h)}\le c,
\end{equation}
\item{(3)} If the pair $(r,\bU)$ belongs to the class (\ref{XT}) there is $c=c(M_0,E_0,\underline r,\overline r, \|\bU, \nabla \bU\|_{L^\infty(Q_T;\Rm^{12})})>0$
such that for all $n=1,\ldots,N$,
\begin{equation}\label{est4}
{\rm sup}_{n=0,\ldots N}{\cal E}(\vr^n,\hat\bu^n|\hat r(t_n),\hat{\vc \bU}(t_n))\le c,
\end{equation}
where
$$
{\cal E}(\vr,\bu|z,\vc v)=\int_{\Omega_h}\Big(\vr|\bu-\vc v|^2+E(\vr|z)\Big){\rm d} x,\;\;
E(\vr|z)=H(\vr)-H'(z)(\vr-z)-H(z).
$$
\item{(4)} There exists $c=c(M_0,E_0,\underline r, |p'|_{C^[\underline r,\overline r]})>0$ such that
\begin{equation}\label{est6}
\deltat\sum_{n=1}^m\sum_{\sigma=K|L \in{\cal E}_{\rm int}}
|\sigma|{(\vr^n_K-\vr^n_L)^2}\Big[\frac{1_{\{\overline\vr_\sigma^n\ge 1\}}}{[{\rm max}\{\vr_K,\vr_L \}]^{2-\gamma}}+1_{\{\overline\vr_\sigma^n< 1\}}\Big]\;|{\bu}^n_\sigma\cdot{\vc n}_{\sigma,K}|\le c\;\quad \mbox{if $\gamma\in [1,2)$},
\end{equation}
$$
\deltat\sum_{n=1}^m\sum_{\sigma=K|L \in{\cal E}_{\rm int}}
|\sigma|{(\vr^n_K-\vr^n_L)^2}\;|{\bu}^n_\sigma\cdot{\vc n}_{\sigma,K}|\le c\;\quad \mbox{if $\gamma\ge 2$}
$$
\end{description}
\end{cor}

\section{Discrete relative energy inequality}
The starting point of our error analysis is the discrete relative energy inequality derived for the numerical scheme (\ref{num1}--\ref{num3}) in \cite[Theorem 5.1]{GHMN}.

\begin{Lemma}\label{Theorem4} 

Let  $(\vr^n,\bu^n)$ be a solution of the discrete problem (\ref{num1}--\ref{num3}) with the pressure $p$ satisfying (\ref{i4}).
Then there holds for all $m=1,\ldots,N$,
\begin{equation}\label{drelativeenergy}
\begin{aligned}
&\sum_{K\in {\cal T}}\frac 12{|K|}\Big(\vr^m_K|{\bu}^m_K-{\bU}^m_{K}|^2-\vr^{0}_K|{\bu}^{0}_K-{\bU}^{0}_{K}|^2\Big)
+ \sum_{K\in {\cal T}}{|K|}\Big(E(\vr_K^m| r_K^m)-E(\vr_K^{0}| r_K^{0})\Big)
\\ &\phantom{\sum_{K\in {\cal T}}}\qquad+\Delta t\sum_{n=1}^m\sum_{K\in {\cal T}}\Big(\mu\int_K|\Grad(\bu^n-\bU^n)|^2 \dx+\frac \mu 3\int_K|{\rm div}(\bu^n-\bU^n)|^2 \dx\Big)
\le \sum_{i=1}^6 T_i,
\end{aligned}
\end{equation}
for any $ 0<r^n\in Q_h(\Omega_h)$, $\bU^n\in V_{h,0}(\Omega_h; \R^3)$, $n=1,\ldots,N$, where
\begin{equation}\label{drelativeenergy*}
\begin{aligned}
&T_1=\Delta t\sum_{n=1}^m\sum_{K\in {\cal T}}\Big(\mu\int_K\Grad\bU^n:\Grad(\bU^n-\bu^n) \dx+\frac \mu 3\int_K{\rm div}\bU^n{\rm div}(\bU^n-\bu^n) \dx\Big),
\\&T_2=\Delta t\sum_{n=1}^m\sum_{K\in{\cal T}}{|K|}\vr_K^{n-1}\frac{{\bU}_{K}^{n}-{\bU}_{K}^{n-1}}{\Delta t}\cdot \Big(\frac{{\bU}_{K}^{n-1} + {\bU}_{K}^{n} }2  -
\bu_K^{n-1}\Big),
\\
&T_3=-\Delta t\sum_{n=1}^m\sum_{K\in{\cal T}}\stik|\sigma|\vr_\sigma^{n,{\rm up}}\Big(\frac{\bU^n_{K}+\bU^n_{L}}2-
\hat{\bu}_{\sigma}^{n,{\rm up}}\Big)\cdot{ {\bU}^n_{K}} [\bu^n_\sigma\cdot\bn_{\sigma,K}],
\\
&T_4=-\Delta t\sum_{n=1}^m\sum_{K\in {\cal T}}\stik |\sigma|p(\vr^n_K)[{\bU}_{\sigma}^{n}\cdot\bn_{\sigma,K}],
\\&T_5=
\Delta t\sum_{n=1}^m\sum_{K\in {\cal T}}\frac{|K|}{\Delta t} ( r^n_K-\vr^n_K)\Big(H'( r^n_K)-H'( r^{n-1}_K)\Big),
\\
&T_6=\Delta t\sum_{n=1}^m\sum_{K\in {\cal T}}\stik |\sigma|\vr_\sigma^{n,{\rm up}}H'( r_K^{ n-1})[\bu^n_\sigma\cdot\bn_{\sigma,K}].
\end{aligned}
\end{equation}
\end{Lemma}

\section{Approximate discrete relative energy inequality}

In this section, we transform the right hand side of the relative energy inequality (\ref{drelativeenergy}) to a form
that is more convenient for the comparison with the strong solution. This transformation is given in the following lemma.

\begin{Lemma}[Approximate relative energy inequality]\label{refrelenergy}
Let  $(\vr^n,\bu^n)$ be a solution of the discrete problem (\ref{num1}--\ref{num3}), where the pressure satisfies (\ref{i4}) with $\gamma\ge 3/2$.
Then there exists
\[
c=c\Big( M_0,E_0,\underline r, \overline r, |p'|_{{C^1}[\underline r,\overline r]}, \|(\partial_t r, \nabla r, \bV,
\partial_t\bV, \nabla\bV)\|_{L^\infty(Q_T;\Rm^{18})},
\]
 \[
 \|\partial_t^2r\|_{L^1(0,T;L^{\gamma'}(\Omega))}, \|\partial_t\nabla r\|_{L^2(0,T;L^{6\gamma/{5\gamma-6}}(\Omega;\Rm^3))}
\Big)>0,
\]
such that for all $m=1,\ldots,N$, we have:
\begin{equation}\label{relativeenergy-}
\begin{aligned}
&\int_{\Omega_h}\Big(\vr^m|\hat{\bu}^m-\hat{\bV}_{h,0}^m|^2+ E(\vr^m| \hat r^m)\Big){\rm d}x
-\int_{\Omega_h}\Big(\vr^{0}|\hat{\bu}^{0}-\hat{\bV}_{h,0}^{0}|^2+E(\vr^{0}|\hat r^{0})\Big){\rm d} x
\\ &
+\Delta t\sum_{n=1}^m\sum_{K\in {\cal T}}\Big(\mu\int_{K}|\Grad(\bu^n-\bV_{h,0}^n)|^2 \dx+\frac \mu 3\int_K|{\rm div}(\bu^n-\bV_{h,0}^n)|^2 \dx\Big)\le \sum_{i=1}^6 S_i+ {  R^m_{h,\Delta t}} { + G^m},
\end{aligned}
\end{equation}
for any couple $(r,\vc V)$ belonging to the class (\ref{XT}), where
\begin{equation}\label{relativeenergy-*}
\begin{aligned}
& S_1= \Delta t \sum_{n=1}^m\sum_{K\in {\cal T}}\Big(\mu\int_{K}\Grad\bV_{h,0}^n:\Grad(\bV_{h,0}^n-\bu^n) \dx+\frac \mu 3\int_{K}{\rm div}\bV_{h,0}^n{\rm div}(\bV_{h,0}^n-\bu^n) \dx\Big),
	\\& S_2=\Delta t\sum_{n=1}^m\sum_{K\in{\cal T}}|K|\vr_K^{n-1}\frac{{\bV}_{h,0,K}^{n}-{\bV}_{h,0,K}^{n-1}}{\Delta t}\cdot \Big({\bV}_{h,0,K}^{n}   -\bu_K^{n}\Big),
	\\ &S_3={  \Delta t\sum_{n=1}^m\sum_{K\in{\cal T}}\sum_{\sigma\in {\cal E}(K)}|\sigma|\vr_\sigma^{n,{\rm up}}
\Big(\hat{\bV}^{n,{\rm up}}_{h,0,\sigma}-\hat{\bu}^{n,{\rm up}}_{\sigma}\Big)\cdot\Big(\bV^n_{h,0,\sigma}-{\bV}^n_{h,0,K}\Big) \hat\bV_{h,0,\sigma}^{n,{\rm up}}\cdot\bn_{\sigma,K}},
%
	\\& S_4=-\Delta t \sum_{n=1}^m\int_{\Omega_h} p(\vr^n)\dv\bV^n\dx,\\
	&S_5=\Delta t \sum_{n=1}^m\int_{\Omega_h} ( \hat r^n-\vr^n)\frac{p'(\hat r^n)}{\hat r^n} [\partial_t r]^n\dx,
	\\ & S_6= -\Delta t\sum_{n=1}^m\int_{\Omega_h}\frac{\vr^n}{\hat r^n} p'(\hat r^n) { \bu^n}\cdot\nabla r^n\dx,	
\end{aligned}
\end{equation}
and
\begin{equation}\label{A1}
{ |G^m|\le c\,\Delta t\sum_{n=1}^m{\cal E}(\vr^n,\hat\bu^n\Big| \hat r^n,  \hat\bV^n),}\;\;
{ |R^m_{h,\Delta t}|\le c (\sqrt{\Delta t} + h^a)},
\end{equation}
with the power $a$  defined in (\ref{M2}) and with the functional ${\cal E}$ introduced in (\ref{est4}).
\end{Lemma}
\noindent
\bProof
\label{6.0}
We take as test functions $\vc U^n=\vc V^n_{h,0}$ and $r^n=\hat r^n$ in the discrete relative energy inequality
(\ref{drelativeenergy}). We keep the left hand side and the first term (term $T_1$) at the right hand side as they stay. The transformation of the remaining terms at the right hand side (terms $T_2-T_6$) is performed in the following steps:
\\ \\
{\bf Step 1:} {\it Term $T_2$.}
\label{6.1}
We have
\begin{equation}\label{T2}
T_2=T_{2,1}+ R_{2,1}+R_{2,2},\mbox{ with }T_{2,1}=\deltat\sum_{n=1}^m\sum_{K\in{\cal T}}|K|\vr_K^{n-1}\frac{{\bV}_{h,0,K}^{n}-{\bV}_{h,0,K}^{n-1}}{\deltat}\cdot \Big({\bV}_{h,0,K
}^{n}   - \bu_K^{n}\Big),
\end{equation}
and
\[
 R_{2,1}=\deltat\sum_{n=1}^m\sum_{K\in{\cal T}}R^{n,K}_{2,1},\;\; R_{2,2}=\deltat\sum_{n=1}^m R_{2,2}^n,
\]
where
$$
 R_{2,1}^{n,K}=
-\frac {|K|}2
\vr_K^{n-1}\frac{({\bV}_{h,0,K}^{n}-{\bV}_{h,0,K}^{n-1})^2}{\deltat}={ -\frac {|K|}2
\vr_K^{n-1}\frac{([{\bV}^{n}-{\bV}^{n-1}]_{h,0,K})^2}{\deltat}},
$$
and
$$
	R^n_{2,2}=
 	- \sum_{K\in{\cal T}}|K|\vr_K^{n-1}\frac{{\bV}_{h,0,K}^{n}-{\bV}_{h,0,K}^{n-1}}{\deltat}\cdot \Big({\bu}_{K}^{n-1}   		 -\bu_{K}^{n}\Big).
$$
We may write by virtue of the first order Taylor formula applied to function $t\mapsto \vc V(t,x)$,
$$
\Big|\frac{[{\bV}^{n}-{\bV}^{n-1}]_{h,0,K}}{\deltat}\Big|=
\Big|\frac 1{|K|}\int_K\Big[\frac 1\deltat \Big[\int_{t_{n-1}}^{t_n} \partial_t\vc V(z,x) {\rm d} z\Big]_{h,0}\Big]{\rm d}x\Big|
$$
$$
=\Big|\frac 1{|K|}\int_K\Big[\frac 1\deltat \int_{t_{n-1}}^{t_n} [\partial_t\vc V(z) \Big]_{h,0}(x){\rm d} z\Big]{\rm d}x\Big|\le \|[\partial_t\vc V ]_{h,0}\|_{L^\infty(0,T;L^\infty(\Omega;\Rm^3))}\le \|\partial_t\vc V \Big\|_{L^\infty(0,T;L^\infty(\Omega;\Rm^3))},
$$
where we have used the property (\ref{ddd}) of the projection $\Pi_{h,0}^V$ on the space $V_{h,0}{ (\Omega_h)}$.
Therefore, thanks to the mass conservation (\ref{masscons}), we get
\begin{equation}\label{R2.1}
|R^{n,K}_{2,1}|\le\frac {M_0} 2|K|\deltat\|\partial_t\bV\|^2_{L^\infty(0,T; L^{\infty}(\Omega;\Rm^3))}.
\end{equation}
To treat term $R^n_{2,2}$ we use the discrete H\"older inequality  and identity (\ref{masscons}) in order to get
\[
|R^n_{2,2}|\le \deltat\; c M_0 \|\partial_t\bV\|^2_{L^\infty(0,T;W^{1,\infty}(\Omega;\Rm^3)}+
c M_0^{1/2}\Big(\sum_{K\in{\cal T}}|K|\vr_K^{n-1}|{\bu}_{K
}^{n-1}   -
\bu_{K}^{n}|^2\Big)^{1/2} \|\partial_t\bV\|_{L^\infty(0,T;L^{\infty}(\Omega;\Rm^3))};
\]
whence, by virtue of  estimate (\ref{denergyinequality}) for the upwind dissipation term (\ref{upwinddissipation_1}), one obtains
\begin{equation}\label{R2.2}
|R_{2,2}|\le \sqrt{\deltat} \,c(M_0, E_0, \|\partial_t\bV\|_{L^\infty(Q_T;\Rm^{3})}).
\end{equation}

\vspace{2mm}\noindent
{\bf Step 2:} {\it Term $T_3$.}
\label{6.2}
Employing the definition (\ref{upwind1}) of upwind quantities, we easily establish that
\begin{align*}
& T_3= T_{3,1} + R_{3,1}, \\
& \mbox{with }T_{3,1}= \deltat\sum_{n=1}^m \sum_{K\in{\cal T}}\sum_{\sigma\in {\cal E}(K)}|\sigma|\vr_\sigma^{n,{\rm up}}\Big(\hat {\bu}_{\sigma}^{n,{\rm up}}-
\hat {\bV}^{n, {\rm up}}_{h,0,\sigma}\Big)\cdot{\bV}^n_{h,0,K} \bu_\sigma^n\cdot\bn_{\sigma,K}, \quad
R_{3,1}=
\deltat\sum_{n=1}^m { \sum_{\sigma \in {\cal E}_{\rm int}}R_{3,1}^{n,\sigma}}, \\
&\mbox{and }{ R_{3,1}^{n,\sigma}}= |\sigma|\vr_K^n \frac{|\bV_{h,0,K}^n-\bV_{h,0,L}^n|^2}2\,[\bu^n_\sigma\cdot\bn_{\sigma,K}]^+
+ |\sigma|\vr_L^n \frac{|\bV_{h,0,L}^n-\bV_{h,0,K}^n|^2}2\,[\bu^n_\sigma\cdot\bn_{\sigma,L}]^+, \; \forall \sigma=K|L \in {\cal E}_{\rm int}.
\end{align*}
{ Writing
$$
\bV^n_{h,0,K}-\bV^n_{h,0,L}= [\bV^n_{h,0}-\bV^n_{h}]_K+\bV^n_{h,K}-\bV^n_h +\bV^n_h-\bV^n_{h,\sigma}
$$
$$
+\bV^n_{h,\sigma}-\bV^n_h+\bV^n_{h}-\bV^n_{h,L}+[\bV^n_h-\bV^n_{h,0}]_L,\; \sigma=K|L\in {\cal E}_{\rm int},
$$
}
and employing estimates (\ref{N2-2}) (if $K\cap\partial\Omega_h=\emptyset$), (\ref{N2-3}) (if $K\cap\partial\Omega_h\neq\emptyset$) to evaluate the $L^\infty$-norm of the first term, (\ref{L2-1}) then (\ref{L1-3})$_{s=1}$ and (\ref{L2-2}) after (\ref{L1-3})$_{s=1}$ to evaluate the $L^\infty$-norm of the second and third terms,  and performing the same tasks at the second line, we get
\begin{equation}\label{help1}
\|\bV^n_{h,0,K}-\bV^n_{h,0,L}\|_{L^{\infty}(K\cup L;\Rm^3)}\leq c h \|\nabla\bV\|_{L^{\infty}(K\cup L;\Rm^9)};
\end{equation}
consequently
\[
	{ |R_{3,1}^{n,\sigma}|}\le h^2\;c \|\nabla\bV\|^2_{L^\infty ((0,T)\times\Omega;\Rm^9)} |\sigma|(\vr^n_K+\vr^n_L) |\bu^n_\sigma|,  \; \forall \sigma=K|L \in {\cal E}_{\rm int},
\]
whence
\begin{equation}\label{R3.1}
\begin{aligned}
  |R_{3,1}| & \le h\;c \|\nabla\bV\|^2_{L^\infty ((0,T)\times\Omega;\Rm^9)}\Big(\sum_{K\in{\cal T}}\sum_{\sigma=K|L\in {\cal E}(K)} h|\sigma|(\vr^n_K+\vr^n_L)^{6/5}\Big)^{5/6} \times\\
&\Big[\deltat \sum_{n=1}^m\Big(\sum_{K\in{\cal T}}\sum_{\sigma\in {\cal E}(K)}h|\sigma||\bu^n_\sigma|^6\Big)^{1/3}\Big]^{1/2}
\le h \; c(M_0,E_0,\|\nabla\bV\|_{L^\infty(Q_T;\Rm^{9})}),
\end{aligned}
\end{equation}
provided $\gamma\ge 6/5$,
thanks to  the discrete H\"older inequality, the equivalence relation (\ref{reg1}), the equivalence of norms (\ref{norms1}) and energy bounds  listed in Corollary \ref{Corollary1}.

Clearly, for each face $\sigma=K|L\in {\cal E}_{\rm int}$,
$
\bu_{\sigma}^n\cdot\bn_{\sigma,K}+\bu_\sigma^n\cdot{\vc n}_{\sigma,L}=0;$ whence, finally
\begin{equation}\label{T3.1}
T_{3,1}= \deltat\sum_{n=1}^m\sum_{K\in{\cal T}}\sum_{\sigma\in {\cal E}(K)}|\sigma|\vr_\sigma^{n,{\rm up}}\Big(\hat
{\bu}_{\sigma}^{n,{\rm up}}-\hat{\bV}^{n,{\rm up}}_{h,0,\sigma}\Big)\cdot\Big({\bV}^n_{h,0,K}-\bV^n_{h,0,\sigma}\Big) \bu_\sigma^n\cdot\bn_{\sigma,K}.
\end{equation}
Before the next transformation of term $T_{3,1}$, we realize that
$$
\vc V^n_{h,0,K}-\vc V^n_{h,0,\sigma}=[\vc V^n_{h,0}-\vc V^n_{h}]_K+ \vc V^n_{h,K}-\vc V^n_h + \vc V^n_{h}-\vc V^n_{h,\sigma}+ [\vc V^n_{h}-\vc V^n_{h,0}]_\sigma;
$$
whence by virtue of (\ref{N2-2}--\ref{N2-3}), (\ref{L2-1}--\ref{L2-2}) and (\ref{L1-3})$_{s=1}$, similarly as
in (\ref{help1}),
\begin{equation}\label{help2}
\|\bV^n_{h,0,K}-\bV^n_{h,0,\sigma}\|_{L^{\infty}(K;\Rm^3)}\leq c h \|\Grad\bV\|_{L^\infty(0,T; L^\infty(\Omega:R^3))},\;\;\sigma\subset K.
\end{equation}
{ Let us now decompose the  term $ T_{3,1}$ as
\begin{equation*}
 \begin{aligned}
&T_{3,1}= T_{3,2}+ R_{3,2}, \mbox{ with } { R_{3,2}=\deltat\sum_{n=1}^mR^{n}_{3,2}}, \\
&T_{3,2}= \deltat\sum_{n=1}^m\sum_{K\in{\cal T}}\sum_{\sigma\in {\cal E}(K)}|\sigma|\vr_\sigma^{n,{\rm up}}
\Big(\hat{\bV}^{n,{\rm up}}_{h,0,\sigma}-\hat{\bu}^{n,{\rm up}}_{\sigma}\Big)\cdot\Big(\bV^n_{h,0,\sigma}-{\bV}^n_{h,0,K}\Big) \hat\bu_\sigma^{n,{\rm up}}\cdot\bn_{\sigma,K},   \mbox{ and }\\
&{ R^{n}_{3,2}}=\sum_{K\in{\cal T}}\sum_{\sigma\in {\cal E}(K)} |\sigma|\vr_\sigma^{n,{\rm up}}
\Big(\hat{\bV}^{n,{\rm up}}_{h,0,\sigma}-\hat{\bu}^{n,{\rm up}}_{\sigma}\Big)\cdot\Big(\bV^n_{h,0,\sigma}-{\bV}^n_{h,0,K}\Big) \Big(\bu_{\sigma}^n-\hat\bu_\sigma^{n,{\rm up}}\Big)\cdot\bn_{\sigma,K}.
 \end{aligned}
\end{equation*}
By virtue of { discrete} H\"older's inequality and estimate (\ref{help2}), we get
\begin{equation*}
 \begin{aligned}
	{ |R^{n}_{3,2}|} & \le
	c \|\nabla\vc  V\|_{L^\infty(Q_T;\Rm^9)}\Big(\sum_{K\in{\cal T}}\sum_{\sigma\in {\cal E}(K)}h|\sigma|
\vr_\sigma^{n,{\rm up}}\Big|\hat\bu_\sigma^{n,{\rm up}}-\hat\bV_{h,0,\sigma}^{n,{\rm up}}\Big|^2\Big)^{1/2}\\
& \times
\Big(\sum_{K\in{\cal T}}\sum_{\sigma\in {\cal E}(K)}h|\sigma|
|\vr_\sigma^{n,{\rm up}}|^{\gamma_0}\Big)^{1/(2{\gamma_0})}\Big(\sum_{K\in{\cal T}}\sum_{\sigma\in {\cal E}(K)}h|\sigma|
\Big|\bu_\sigma^{n}-\hat\bu_\sigma^{n,{\rm up}}\Big|^q\Big)^{1/q},
 \end{aligned}
\end{equation*}
where $\frac 12+\frac 1{2\gamma_0}+\frac 1q=1$, $\gamma_0={\rm min}\{\gamma, 2\}$ and $\gamma\ge 3/2$. For the sum in the last term of the above product, we have
$$
\sum_{K\in{\cal T}}\sum_{\sigma\in {\cal E}(K)}h|\sigma|
\Big|\bu_\sigma^{n}-\hat\bu_\sigma^{n,{\rm up}}\Big|^q\le c \sum_{K\in{\cal T}}\sum_{\sigma\in {\cal E}(K)}h|\sigma|
|\bu_\sigma^{n}-\bu_K^{n}|^q
$$
$$
\le c \Big(\sum_{K\in {\cal T}}\sum_{\sigma\in {\cal E}(K)}
\Big(\|\bu_\sigma^{n}-\bu^n\|_{L^q(K;\Rm^3)}^q+
\sum_{K\in {\cal T}}
\|\bu^n-\bu_K^{n}\|_{L^q(K;\Rm^3)}^q\Big)
\le c h^{\frac {2\gamma_0-3}{2\gamma_0}q}\Big(\sum_{K\in {\cal T}}\|\nabla_x\bu^n\|^2_{L^2(K;\Rm^9}\Big)^{q/2},
$$
where we have used the definition (\ref{upwind1}), the discrete Minkowski inequality, interpolation inequalities
(\ref{interpol1}--\ref{interpol2}) and the discrete `imbedding' inequality (\ref{dod1*}). Now we can go back to the estimate of $R_{3,2}^n$ taking into account the upper bounds
(\ref{est0}), (\ref{est3}--\ref{est4}), in order to get
 \begin{equation}\label{R3.2}
 |R_{3,2}|\le h^a\;c (M_0,E_0,\|\nabla\vc  V\|_{L^\infty(Q_T;\Rm^9)}),
 \end{equation}
 provided $\gamma\ge 3/2$, where $a$ is given in (\ref{A1}).

 Finally, we rewrite term $T_{3,2}$ as
 \begin{equation}\label{T3}
 \begin{aligned}
&T_{3,2}= T_{3,3}+ R_{3,3}, \mbox{ with } { R_{3,3}=\deltat\sum_{n=1}^mR^{n}_{3,3}}, \\
&T_{3,3}= \deltat\sum_{n=1}^m\sum_{K\in{\cal T}}\sum_{\sigma\in {\cal E}(K)}|\sigma|\vr_\sigma^{n,{\rm up}}
\Big(\hat{\bV}_{h,0,\sigma}^{n,{\rm up}}-\hat{\bu}^{n,{\rm up}}_{\sigma}\Big)\cdot\Big(\bV^n_{h,0,\sigma}-{\bV}^n_{h,0,K}\Big) \hat\bV_{h,0,\sigma}^{n,{\rm up}}\cdot\bn_{\sigma,K},   \mbox{ and }\\
&{ R^{n}_{3,3}}=\sum_{K\in{\cal T}}\sum_{\sigma\in {\cal E}(K)} |\sigma|\vr_\sigma^{n,{\rm up}}
\Big(\hat{\bV}_{h,0,\sigma}^{n,{\rm up}}-\hat{\bu}^{n,{\rm up}}_{\sigma}\Big)\cdot\Big(\bV^n_{h,0,\sigma}-{\bV}^n_{h,0,K}\Big) \Big(\hat\bu_\sigma^{n,{\rm up}}-\hat \bV_{h,0,\sigma}^{n,{\rm up}}\Big)\cdot\bn_{\sigma,K};
 \end{aligned}
\end{equation}
whence
\begin{equation}\label{R3.3}
|R_{3,3}|\le c(\|\nabla\bV\|_{L^\infty(Q_T,\Rm^9)})\; \deltat\sum_{n=1}^m{\cal E}(\vr^n,\hat\vu^n\,|\,\hat r^n,\hat \bV_{h,0}^n).
\end{equation}

 }

\vspace{2mm}\noindent
{\bf Step 3:} {\it Term $T_4$.}
\label{6.3}
Integration  by parts over each $K\in {\cal T}$ gives
$$
T_4=-\deltat \sum_{n=1}^m\sum_{K\in {\cal T}}{{\rm \int_K}} p(\vr_K^n){\rm div}_x\bV_{h,0}^n\dx .
$$
 We may write
 \begin{equation}\label{help3}
 \|{\rm div}_x(\bV^n_{0,h}-\bV^n_h)\|_{L^\infty(K)}\leq  c h \|\Grad \bV\|_{L^\infty(0,T;L^\infty(\Omega;\Rm^9))},
 \end{equation}
 where we have used (\ref{N2-2}--\ref{N2-3}). Therefore, employing identity  (\ref{L1-1}) we obtain
 \begin{equation}\label{T4}
 T_4=T_{4,1} + R_{4,1},\quad T_{4,1}= -\deltat \sum_{n=1}^m\sum_{K\in {\cal T}}{{\rm \int_K}} p(\vr_K^n){\rm div}_x\bV^n\dx,
 \end{equation}
 $$
 R_{4,1}= -\deltat \sum_{n=1}^m\sum_{K\in {\cal T}}{{\rm \int_K}} p(\vr_K^n){\rm div}_x(\bV^n_{h,0}-\bV_{h}^n)\dx.
 $$
Due to (\ref{i4}) and (\ref{est3}), $p(\vr^n)$ is bounded uniformly in $L^\infty(L^1(\Omega))$; employing this fact and (\ref{help3}) we immediately get
\begin{equation}\label{R4.1}
|R_{4,1}|\le h\; c(E_0, M_0,\|\nabla \vc V\|_{L^\infty(0,T;L^\infty(\Omega;\R^9))}).
\end{equation}

\vspace{2mm}\noindent
{\bf Step 4:} {\it Term $T_5$.}
 \label{6.4}
Using the Taylor formula, we get
\[
H'(r_K^n)-H'(r_K^{n-1})=H''(r_K^{n})(r_K^n-r_K^{n-1}) -\frac 12H'''(\overline r_K^n)(r_K^n-r_K^{n-1})^2,
\]
where $\overline r_K^n\in[\min(r_K^{n-1},r_K^n), \max(r_K^{n-1},r_K^n)]$;
we infer
\begin{equation*}
\begin{aligned}
	& T_5= T_{5,1}+ R_{5,1},\mbox{ with }  T_{5,1}=\deltat \sum_{n=1}^m\sum_{K\in {\cal T}}|K| ( r^n_K-\vr^n_K)\frac{p'(r_K^n)}{r_K^n} \frac{r_K^n-r_K^{n-1}}{\deltat}, \,  R_{5,1}=\deltat \sum_{n=1}^m\sum_{K\in {\cal T}} R_{5,1}^{n,K}, \mbox{ and } \\
	& R_{5,1}^{n,K}= \frac 12|K|H'''(\overline r_K^n)\frac{(r_K^n-r_K^{n-1})^2}{\deltat}(\vr_K^n-r_K^n).
 \end{aligned}
\end{equation*}
Consequently, by the { first order Taylor formula applied to function $t\mapsto r(t,x)$ on the interval $(t_{n-1}, t_n)$} and thanks to the mass conservation \eqref{masscons}
\begin{equation}\label{R5.1}
	|R_{5,1}|\le \deltat \;c(M_0,\underline r,\overline r, |p'|_{C^1([\underline r,\overline r]},\|\partial_t r\|_{L^\infty(Q_T)}).
\end{equation}

\vspace{2mm}
Let us now decompose  $T_{5,1}$ as follows:
\begin{equation}\label{T5}
\begin{aligned}
	& T_{5,1}=T_{5,2}+ R_{5,2}, \mbox{ with }T_{5,2}=\deltat \sum_{n=1}^m\sum_{K\in {\cal T}}{ \int_K} ( r^n_K-\vr^n_K)\frac{p'(r_K^n)}{r_K^n} [\partial_t r]^n { {\rm d}x}, \,  R_{5,2}=\deltat \sum_{n=1}^m\sum_{K\in {\cal T}} R_{5,2}^{n,K}, \mbox{ and} \\
	& R_{5,2}^{n,K}= {\int_K} ( r^n_K-\vr^n_K)\frac{p'(r_K^n)}{r_K^n}\Big(\frac{r_K^n-r_K^{n-1}}{\deltat} -[\partial_t r]^n\Big){ {\rm d} x}. \end{aligned}
\end{equation}
In accordance with (\ref{n2}), here and in the sequel, $[\partial_t r]^n(x)=\partial_t r(t_n,x)$.
{ We write using twice the Taylor formula in the integral form and the Fubini theorem,
$$
|R_{5,2}^{n,K}|= \frac 1\deltat\Big|{p'(r^n_K)}{r^n_K} (\vr^n_K-r^n_K)\int_K\int^{t_n}_{t_{n-1}}\int_s^{t_n}\partial_t ^2 r(z){\rm d}z{\rm d}s {\rm d}x\Big|
$$
$$
\le \frac {p'(r^n_K)}{r^n_K}\int^{t_n}_{t_{n-1}}\int_K|\vr^n_K-r^n_K|\Big|\partial_t ^2 r(z)\Big|{\rm d}x{\rm d}z{\rm d}s
$$
$$
\le \frac{p'(r^n_K)}{r^n_K}
\|\vr^n-\hat r^n\|_{L^{\gamma}(K)}\int^{t_n}_{t_{n-1}}\|\partial_t^2 r(z)\|_{L^{\gamma'}(K)}{\rm d} z{\rm d }s.
$$

Therefore, by virtue of Corollary \ref{Corollary1}, we have estimate
\begin{equation}\label{R5.2}
	|R_{5,2}|\le \deltat\; c(M_0, E_0,\underline r,\overline r, |p'|_{C^1([\underline r,\overline r]},\|\partial^2_t r\|_{L^1(0,T; L^{\gamma'}(\Omega)}).
\end{equation}
}



\vspace{2mm}\noindent
{\bf Step 5:} {\it Term $T_6$.}
%
We decompose this term as follows:
$$
 \begin{aligned}
	  &T_6=T_{6,1} + R_{6,1},\quad  R_{6,1}=\deltat \sum_{n=1}^m \sum_{K\in {\cal T}}\sum_{\sigma\in {\cal E}(K)}R_{6,1}^{n,\sigma,K}, \mbox{ with} \\
	&T_{6,1}=\deltat\sum_{n=1}^m\sum_{K\in {\cal T}}\sum_{\sigma=K|L\in {\cal E}(K)}|\sigma|\vr_K^{n}\Big( H'( r_K^{ {n-1}})-H'(r_\sigma^{ {n-1}})\Big)\bu_\sigma^n\cdot\bn_{\sigma,K}, \mbox{ and} \\
	&  R_{6,1}^{n,\sigma,K}=|\sigma|\Big(\vr_\sigma^{n,{\rm up}}-\vr_K^{n}\Big)\Big( H'( r_K^{ {n-1}})-H'(r_\sigma^{ {n-1}})\Big)\bu_\sigma^n\cdot\bn_{\sigma,K}, { \mbox{ for } \sigma=K|L\in {\cal E}_{\rm int}.}
 \end{aligned}
$$
We will now estimate the term $R_{6,1} ^{n,\sigma,K}$. We shall treat separately the cases $\gamma<2$ and $\gamma\ge 2$. The `simple' case $\gamma\ge 2$ is left to the reader. The more complicated case $\gamma<2$ will be treated as follows: We first write
$$
	 |R_{6,1} ^{n,\sigma,K}| \le \sqrt h\, \|\nabla H'(r)\|_{L^\infty(Q_T;\Rm^3)} |\sigma| |\vr_\sigma^{n,{\rm up}}-\vr_K^{n}|\Big[\frac {1_{\{\overline\vr^n_\sigma\ge 1\}}}{[{\rm max}\{\vr_K,\vr_L\}]^{(2-\gamma)/2}}+1_{\{\overline\vr^n_\sigma< 1\}}\Big]\,\sqrt{|\bu_\sigma^n\cdot\bn_{\sigma,K}|}\times
$$
$$
\Big[{1_{\{\overline\vr^n_\sigma\ge 1\}}}{[{\rm max}\{\vr_K,\vr_L\}]^{(2-\gamma)/2}}+1_{\{\overline\vr^n_\sigma< 1\}}\Big]
\sqrt h \sqrt{|\bu_\sigma^n\cdot\bn_{\sigma,K}|} ,
$$
where we have  employed the  first order Taylor formula applied to function $x\mapsto H'(r(t_{n-1},x)$.
Consequently, the application of the discrete H\"older and Young inequalities yield
$$
 \begin{aligned}
 |R_{6,1}|
	& \le \sqrt h\,  { c} \|\nabla H'(r)\|_{L^\infty(Q_T;\Rm^3)}\times\\
		&\deltat \sum_{n=1}^m
 \Big(\sum_{K\in {\cal T}}\sum_{\sigma\in {\cal E}(K)}|\sigma| h
 \Big[{1_{\{\overline\vr^n_\sigma\ge 1\}}}{[{\rm max}\{\vr_K,\vr_L\}]^{2-\gamma}}+1_{\{\overline\vr^n_\sigma< 1\}}\Big]
 \,|\bu_\sigma^n\cdot\bn_{\sigma,K}|\Big)^{1/2}\times\\
&\Big(\sum_{K\in {\cal T}}\sum_{\sigma=K|L\in {\cal E}(K)}|\sigma|h (\vr_\sigma^{n,{\rm up}}-\vr_K^{n})^2\Big[\frac {1_{\{\overline\vr^n_\sigma\ge 1\}}}{[{\rm max}\{\vr_K,\vr_L\}]^{2-\gamma}}+1_{\{\overline\vr^n_\sigma< 1\}}\Big] \,|\bu_\sigma^n\cdot\bn_{\sigma,K}| \Big)^{1/2}\\
& \le \sqrt h\,  { c} \|\nabla H'(r)\|_{L^\infty(Q_T;\Rm^3)}\times\\
		&\deltat \sum_{n=1}^m\Big\{ \Big[|\Omega_h|^\frac 56+
 \Big(\sum_{K\in {\cal T}}|\sigma| h(\vr^n_K)^{\frac 65(2-\gamma)}\Big)^{\frac 56}\Big]\;\Big(
 \,\sum_{K\in {\cal T}}\sum_{\sigma\in {\cal E}(K)}|\sigma|h|\bu_\sigma^n\cdot\bn_{\sigma,K}|^6 \Big)^{\frac 16}
 \\
&\sum_{K\in {\cal T}}\sum_{\sigma=K|L\in {\cal E}(K)}|\sigma|h (\vr_\sigma^{n,{\rm up}}-\vr_K^{n})^2\Big[\frac {1_{\{\overline\vr^n_\sigma\ge 1\}}}{[{\rm max}\{\vr_K,\vr_L\}]^{2-\gamma}}+1_{\{\overline\vr^n_\sigma< 1\}}\Big]|\bu_\sigma^n\cdot\bn_{\sigma,K}\Big\}^{1/2}
\\
& \le \sqrt h\, { c}   \|\nabla H'(r)\|_{L^\infty(Q_T;\Rm^3)} \Big\{
\;\deltat\sum_{n=1}^m\Big[|\Omega_h|^\frac 56+
 \Big(\sum_{K\in {\cal T}}|\sigma| h(\vr^n_K)^{\frac 65(2-\gamma)}\Big)^{\frac 56}\Big]
\Big(\sum_{\sigma\in {\cal E}}|\sigma| h|\bu_\sigma^n|^6\Big)^{1/6}
 \\
 & +\deltat\sum_{n=1}^m\Big[\sum_{K\in {\cal T}}\sum_{\sigma=K|L\in {\cal E}(K)}|\sigma|h (\vr_\sigma^{n,{\rm up}}-\vr_K^{n})^2
 \Big[\frac {1_{\{\overline\vr^n_\sigma\ge 1\}}}{[{\rm max}\{\vr_K,\vr_L\}]^{2-\gamma}}+1_{\{\overline\vr^n_\sigma< 1\}}\Big]
 \,|\bu^n_\sigma\cdot\bn_{\sigma,K}|  \Big]\Big\}
	\\  & \le \sqrt h \; c(M_0,E_0,\underline r,\overline r, |p'|_{C([\underline r,\overline r])}, \|\nabla r\|_{L^\infty(Q_T;\Rm^3)}),
\end{aligned}
$$
 where, in order to get the last line, we have used the estimate (\ref{est6}) of the numerical dissipation  to evaluate the second term, and finally
equivalence of norms (\ref{norms1})$_{p=6}$ together with (\ref{est1}) and (\ref{est3}), under assumption $\gamma\ge 12/11$, to evaluate the first term.

\vspace{2mm}
Let us now decompose the term $T_{6,1}$ as
{
\begin{equation*}
\begin{aligned}
	&T_{6,1}=T_{6,2}+ R_{6,2},  \mbox{ with }  T_{6,2}=\deltat\sum_{n=1}^m\sum_{K\in{\cal T}}\sum_{\sigma=K|L\in {\cal E}(K)}
|\sigma|\vr^n_K H''(r_K^{ n-1})(r_K^{ n-1}-r_\sigma^{ n-1}) [\bu^n_\sigma\cdot\bn_{\sigma,K}],
 \\    &  R_{6,2}=\deltat\sum_{n=1}^m\sum_{K\in{\cal K}}\sum_{\sigma\in {\cal E}(K) }
R_{6,2}^{n,\sigma,K},\;
\mbox{and }\\ & R_{6,2}^{n,\sigma,K}=|\sigma|\vr^n_K \Big(H'(r^{n-1}_K)-H'(r^{n-1}_\sigma)-H''(r^{n-1}_K)(r^{n-1}_K-r^{n-1}_\sigma)\Big)
[\bu^n_\sigma\cdot\bn_{\sigma,K}].
 \end{aligned}
\end{equation*}
}
Therefore, by virtue of the { second order Taylor formula applied to function $H'$}, the H\"older inequality,  (\ref{norms1}),
and (\ref{est1}), (\ref{est3}) in Corollary \ref{Corollary1}, we have, provided
$\gamma\ge 6/5$,
{
\begin{align}
| R_{6,2}| &\le h c \Big(|H''|_{C([\underline r,\overline r])}+|H'''|_{C([\underline r,\overline r])}\Big)\|\nabla r\|_{L^\infty(Q_T;\Rm^3)}\,\Delta t \sum_{n=1}^m\|\vr^n\|_{L^\gamma(\Omega_h)} \|\bu^n\|_{L^6(\Omega_h;\Rm^3)}
\nonumber \\
\label{R6.2}
& \le  h\;
c(M_0,E_0,\underline r, \overline r, |p'|_{C^1([\underline r,\overline r])},  \|\nabla r\|_{L^\infty(Q_T;\Rm^3)} ).
\end{align}
}

\vspace{2mm}

Let us now deal with the term  $T_{6,2}$.
Noting that $ \displaystyle
\int_K\nabla r^{ n-1} \dx =   \sum_{\sigma\in {\cal E}(K)} |\sigma|(r_\sigma^{ n-1}-r_K^{ n-1})\bn_{\sigma,K},$ we may write
$T_{6,2}= T_{6,3}+ R_{6,3},$ with
$$
\begin{aligned}
& T_{6,3}=
 -\deltat\sum_{n=1}^m\sum_{K \in {\cal T}}\int_K{\vr^n_K} H''(r_K^{ n-1}) \bu^n \cdot\nabla r^{ n-1}\dx, \\
&  R_{6,3} = \deltat\sum_{n=1}^m\sum_{K \in {\cal T}} \int_K\vr^n_K H''(r_K^{ n-1})(\bu^n- \bu^n_K)\cdot\nabla r^{ n-1}\dx
 \\ & \hspace{3cm}+\deltat\sum_{n=1}^m\sum_{K \in {\cal T}}\sum_{\sigma\in {\cal E}(K)} |\sigma|\vr^n_K H''(r_K^{ n-1})(r_K^{ n-1}-r_\sigma^{ n-1})
(\bu^n_\sigma-\bu^n_K)\cdot\bn_{\sigma,K}.
\end{aligned}
$$
Consequently, by virtue of H\"older's inequality, interpolation inequality (\ref{interpol1}) (to estimate
$\|\vc u^n-\vc u^n_K\|_{L^{\gamma_0'}(K;\Rm^3)}$ by $h^{(5\gamma_0-6)/(2\gamma_0)}\|\Grad\vu^n \|_{L^2(K;\Rm^9)}$, $\gamma_0=\min\{\gamma,2\}$)
in the first term, and by the Taylor formula applied to function $x\mapsto r(t_{n-1},x)$, then H\"older's inequality
 and (\ref{interpol1}--\ref{interpol2}) (to estimate $\|\vc u_\sigma^n-\vc u^n_K\|_{L^{\gamma_0'}(K;\Rm^3)}$ by $h^{(5\gamma_0-6)/(2\gamma_0)}\|\Grad\vu^n \|_{L^2(K;\Rm^9)}$), we get
\begin{equation}\label{R6.3}
|R_{6,3}|\le h^b \;c(M_0,E_0,\underline r, \overline r, |p'|_{C^1([\underline r,\overline r])} \|\nabla r\|_{L^\infty(Q_T;\Rm^3)} ),\;\; b=\frac{5\gamma_0-6}{2\gamma_0},
\end{equation}
provided $\gamma\ge 6/5$, where we have used at the end the discrete imbedding and  H\"older inequalities
(\ref{dod1*}--\ref{dod2*}) and finally estimates (\ref{est0}) and (\ref{est3}).

{
Finally we write
 $T_{6,3}= T_{6,4}+ R_{6,4},$ with
\begin{equation}\label{T6}
\begin{aligned}
& T_{6,4}=
 -\deltat\sum_{n=1}^m\sum_{K \in {\cal T}}\int_K{\vr^n_K}\frac{ p'(r_K^{n})}{r_K^n} \bu^n\cdot\nabla r^{ n}\dx, \\
&  R_{6,4} = \deltat\sum_{n=1}^m\sum_{K \in {\cal T}} \int_K\vr^n_K \Big(H''(r_K^{ n})\nabla r^{n}-
H''(r_K^{n-1})\nabla r^{n-1}\Big)\cdot\bu^n\dx,
\end{aligned}
\end{equation}
{ where by the same token as in (\ref{R5.2}),
\begin{equation}\label{R6.4}
|R_{6,4}|\le \deltat\; c (M_0,E_0, \underline r,\overline r, |p'|_{C^1([\underline r,\overline r])}, \|\nabla r, \partial_t r\|_{L^\infty(Q_T;\Rm^4)}, \|\partial_t\nabla r\|_{L^2(0,T; L^{{6\gamma}/{(5\gamma-6)}}(\Omega;\Rm^3))}),
\end{equation}
provided $\gamma\ge 6/5$.
}
}

\vspace{2mm}
We are now in position to conclude the proof of Lemma \ref{refrelenergy}: we obtain the inequality \eqref{relativeenergy-} by gathering the principal terms (\ref{T2}), (\ref{T3}), (\ref{T4}), (\ref{T5}), (\ref{T6}) and  the residual terms estimated in   (\ref{R2.1}), (\ref{R2.2}), (\ref{R3.1}), (\ref{R3.2}), (\ref{R3.3}), (\ref{R5.1}), (\ref{R5.2}),   (\ref{R6.2}), (\ref{R6.3}), { (\ref{R6.4}) at}  the right hand side $\sum_{i=1}^6 T_i$  of the discrete relative energy inequality (\ref{drelativeenergy}).

\qed

\section{ A discrete identity satisfied by  the strong solution}\label{7}
{ This section is devoted to the proof of a discrete identity satisfied by any strong solution of problem (\ref{i1}--\ref{i7}) in the class (\ref{r5}--\ref{r6}) extended eventually to $\R^3$ according to Lemma \ref{LEx1}. This identity
is stated in Lemma \ref{strongentropy} below. It will be used in combination with the approximate relative energy inequality stated in Lemma \ref{refrelenergy} to deduce the convenient form of the relative energy inequality verified by any function being a strong solution to the compressible Navier-Stokes system. This last step is performed in the next section.}

\begin{Lemma}[A discrete identity for strong solutions]\label{strongentropy}
Let  $(\vr^n,\bu^n)$ be a solution of the discrete problem (\ref{num1}--\ref{num3}) with the pressure satisfying (\ref{i4}), where $\gamma\ge 3/2$.
There exists
\[
c=c\Big(M_0,E_0,\underline r, \overline r, |p'|_{{C^1}[\underline r,\overline r]}, \|(\partial_t r, \nabla r, \bV,
\partial_t\bV, \nabla\bV, \nabla^2\bV)\|_{L^\infty(Q_T;\Rm^{45})},
\]
 \[
 \|\partial_t^2r\|_{L^1(0,T;L^{\gamma'}(\Omega))}, \|\partial_t\nabla r\|_{L^2(0,T;L^{6\gamma/{5\gamma-6}}(\Omega;\Rm^3))}, \|\partial_t^2\bV,\partial_t\nabla\bV\|_{L^2(0,T;L^{6/5}(\Omega;\Rm^{12}))}
\Big)>0,
\]
such that for all $m=1,\ldots,N$, we have:
\begin{equation}\label{relativeenergy-**}
 \sum_{i=1}^6 {\cal S}_i+ {  {\cal R}^m_{h,\Delta t}} =0,
\end{equation}
 where
$$
\begin{aligned}
&{\cal  S}_1= \Delta t \sum_{n=1}^m\sum_{K\in {\cal T}}\Big(\mu\int_{K}\Grad\bV_{h,0}^n:\Grad(\bV_{h,0}^n-\bu^n) \dx+\frac \mu 3\int_{K}{\rm div}\bV_{h,0}^n{\rm div}(\bV_{h,0}^n-\bu^n) \dx\Big),
	\\& {\cal S}_2=\Delta t\sum_{n=1}^m\sum_{K\in{\cal T}}|K|r_K^{n-1}\frac{{\bV}_{h,0,K}^{n}-{\bV}_{h,0,K}^{n-1}}{\Delta t}\cdot \Big({\bV}_{h,0,K}^{n}   -\bu_K^{n}\Big),
	\\ &{\cal S}_3={  \Delta t\sum_{n=1}^m\sum_{K\in{\cal T}}\sum_{\sigma\in {\cal E}(K)}|\sigma|r_\sigma^{n,{\rm up}}
\Big(\hat{\bV}^{n,{\rm up}}_{h,0,\sigma}-\hat{\bu}^{n,{\rm up}}_{\sigma}\Big)\cdot\Big(\bV^n_{h,0,\sigma}-{\bV}^n_{h,0,K}\Big) \hat\bV_{h,0,\sigma}^{n,{\rm up}}\cdot\bn_{\sigma,K}}
%
	\\& {\cal S}_4=-\Delta t \sum_{n=1}^m\int_{\Omega_h} p(\hat r^n)\dv\bV^n\dx,\\
&{\cal S}_5=0,
	\\ &{\cal  S}_6= -\Delta t\sum_{n=1}^m\int_{\Omega_h} p'(\hat r^n) { \bu^n}\cdot\nabla r^n\dx,	
\end{aligned}
$$
and
\[
|{\cal R}_{h, \deltat}^m|\le c\Big( h^{5/6} + \deltat\Big),
\]
for any couple $(r,\vc V)$ belonging to (\ref{XT}) and satisfying the continuity equation
(\ref{i1}) on $(0,T)\times\R^3$ and momentum equation (\ref{i2}) with boundary conditions (\ref{i6}) on $(0,T)\times\Omega$ in the classical sense.
\end{Lemma}
{ Before starting the proof  we recall an auxiliary algebraic inequality whose straightforward proof is left to the reader, and introduce some notations.}
\begin{Lemma}\label{LL1}
	Let $p$ satisfies assumptions(\ref{i4}).
	Let $0<a<b<\infty$. Then there exists $c=c(a,b)>0$ such that for all $\vr\in [0,\infty)$ and $r\in [a,b]$ there holds
	\[
		E(\vr|r)\ge c(a,b)\Big(1_{R_+\setminus [a/2,2 b]}{ (\vr)}+\vr^\gamma 1_{R_+\setminus [a/2,2 b]}{ (\vr)}+ (\vr-r)^2 1_{[a/2,2 b]}{ (\vr)}\Big),
	\]
	where $E(\vr|r)$ is defined in (\ref{est4}).
\end{Lemma}
{ If we take in Lemma \ref{LL1} $\vr=\vr^n(x)$, $r=\hat r^n(x)$, $a=\underline r$, $b=\overline r$ (where r is a function belonging to class (\ref{XT}) and $\underline r$, $\overline r$
are its lower and upper bounds, respectively), we obtain
\begin{equation}\label{added}
E(\vr^n(x)|\hat r^n(x))\ge c(\underline r,\overline r)\Big(1_{R_+\setminus [\underline r/2,2 \overline r]}{ (\vr^n(x))}+(\vr^n)^\gamma(x) 1_{R_+\setminus [\underline r/2,2 \overline r]}{ (\vr^n(x))}+ (\vr^n(x)-\hat r^n(x))^2 1_{[\underline r/2,2 \overline r]}{ (\vr^n(x))}\Big).
\end{equation}
}
 Now,  for fixed numbers $\underline r$ and $\overline r$  { and  fixed functions $\vr^n$, $n=0,\ldots, N$,  we introduce the residual and essential subsets  of $\Omega$ (relative to $\vr^n$) as follows:}
\begin{equation}\label{essres}
N_{\rm ess}^n=\{x\in\Omega\,\Big|\,\frac 12\underline r\le \vr^n(x)\le 2\overline r\},\;
N_{\rm res}^n= \Omega\setminus N_{\rm ess}^n,
\end{equation}
and we set
\[
[g]_{\rm ess}{ (x)}= g{ (x)} 1_{N^n_{\rm ess}}{ (x)},\; [g]_{\rm res}{ (x)}= g { (x)}1_{N^n_{\rm res}}{ (x)},\;\; { x\in \Omega},\;\;g\in L^1(\Omega).
\]

 Integrating inequality (\ref{added}) we deduce
\begin{equation}\label{rentropy}
c(\underline r,\overline r)\sum_{K\in{\cal T}} \int_K{\Big(\Big[1\Big]_{\rm res}+\Big[(\vr^n)^\gamma\Big]_{\rm res}+\Big[\vr^n-\hat r^n\Big]^2_{\rm ess}\Big)}\dx\le{\cal E}(\vr^n,\bu^n\Big| \hat r^n,\bV^n),
\end{equation}
for any  pair $(r,\bV)$ belonging to the class (\ref{XT}) and any $\vr^n\in Q_h(\Omega_h)$, $\vr^n\ge 0$.

{ We are now ready to proceed to the proof of Lemma \ref{strongentropy}.}
\\ \\
\bProof
Since $(r,\bV)$ satisfies  \eqref{i1} on $(0,T)\times \Omega$ and belongs to the class (\ref{XT}), Equation (\ref{i2}) can be rewritten in the form
$$
r\partial_t\bV+r\bV\cdot\nabla\bV +\nabla p(r)-\mu\Delta\bV -\mu/3\nabla\dv\bV=0 \;\;\mbox{in $(0,T)\times\Omega$}.
$$
From this fact, we deduce the identity
\begin{equation}\label{strong0}
\sum_{i=1}^5{\cal T}_i={\cal R}_0,
\end{equation}
where
$$
\begin{aligned}
&  {\cal R}_0=\deltat\sum_{n=1}^m\int_{\Omega_h\setminus\Omega}\Big(
r^n[\partial_t\bV]^n+r\bV^n\cdot\nabla\bV^n +\nabla p(r^n)-\mu\Delta\bV^n -\frac\mu 3\nabla\dv\bV^n\Big)\cdot(\bV_{h,0}^n-\vu^n){\rm d}x,\\
&{\cal T}_1 = -\deltat\sum_{n=1}^m\int_{\Omega_h}\Big(\mu\Delta \bV^n+ \frac \mu 3
\nabla\dv\bV^n\Big)\cdot(\bV_{h,0}^n-\bu^n)\dx, \quad
{\cal T}_2 =\deltat\sum_{n=1}^m\int_{\Omega_h} r^{n}[\partial_t\bV]^n\cdot (\bV_{h,0}^n-\bu^n)\dx,\\
& {\cal T}_3 = \deltat \sum_{n=1}^m\int_{\Omega_h} r^n\bV^n\cdot\nabla\bV^n\cdot (\bV_{h,0}^n-\bu^n)\dx, \quad
{\cal T}_4 =\deltat\sum_{n=1}^m\int_{\Omega_h}\nabla p(r^n)\cdot\bV^n_{h,0}\dx,\\
& {\cal T}_5=0,
\quad\quad
{\cal T}_6 = -\deltat\sum_{n=1}^m\int_{\Omega_h}\nabla p(r^n)\cdot\bu^n \dx.
 \end{aligned}
 $$

 In the steps below, we deal with each of the terms ${\cal R}_0$ and ${\cal T}_i$.
 \\ \\\noindent
 \textbf{Step 0: }\textit{Term ${\cal R}_0$.}
 By the H\"older inequality
 $$
 |{\cal R}_0|\le |\Omega_h\setminus\Omega|^{5/6}\, c(\overline r, |p'|_{C[\underline r,\overline r]},\|(\partial_t r, \nabla r,\vc V, \nabla\vc V,
 \nabla^2\vc V)\|_{L^\infty(Q_T;\Rm^{43})}\deltat\sum_{n=1}^m(\|\bu^n\|_{L^6(\Omega_h)}+\|\bV^n_{h,0}\|_{L^6(\Omega_h)})
 $$
 \begin{equation}\label{cR0}
 \le h^{5/3}\, c(M_0,E_0,\overline r, |p'|_{C[\underline r,\overline r]}, \|(\partial_t r, \nabla r,\vc V, \nabla\vc V,
 \nabla^2\vc V)\|_{L^\infty(Q_T;\Rm^{43})},
 \end{equation}
 where we have used (\ref{est1}) and (\ref{N2-2}--\ref{N2-3}), (\ref{ddd}).

\vspace{2mm}
\textbf{Step 1: }\textit{Term ${\cal T}_1$.}\label{7.3}
 Integrating by parts, we get:
 {
\begin{equation}\label{cT1}
 \begin{aligned}
&{\cal T}_1={\cal T}_{1,1} + {\cal R}_{1,1},
\\
& \mbox{with }{\cal T}_{1,1} = \deltat\sum_{n=1}^m\sum_{K\in {\cal T}}\int_K\Big(\mu\nabla\bV_{h,0}^n:\nabla(\bV_{h,0}^n-\bu^n)+
\frac \mu 3\dv \bV_{h,0}^n\dv (\bV_{h,0}^n-\bu^n)\Big)\dx,
 \\ &\mbox{and }\;
 {\cal R}_{1,1}=I_1+I_2,\;\mbox{with}\\ & I_1=
 \deltat\sum_{n=1}^m\sum_{K\in {\cal T}}\int_K\Big(\mu\nabla(\bV^n-\bV_{h,0}^n):\nabla(\bV_{h,0}^n-\bu^n)
 + \frac \mu 3\dv(\bV^n-\bV_{h,0}^n)\dv(\bV_{h,0}^n-\bu^n)\Big)
 \dx,
 \\
 &
 I_2=-\deltat\sum_{n=1}^m \sum_{K\in{\cal T}}\sum_{\sigma\in {\cal E}(K)} \int_{\sigma} \Big(\mu\bn_{\sigma,K} \cdot\nabla\bV^n \cdot(\bV_{h,0}^n-\bu^n) + \frac\mu 3\dv\bV^n(\bV^n_{h,0}-\bu^n)\cdot\bn_{\sigma,K}\Big)\dS
 \\
 & =
 -\deltat\sum_{n=1}^m\ \sum_{\sigma\in {\cal E}} \int_{\sigma} \Big(\mu\bn_{\sigma} \cdot\nabla\bV^n \cdot\Big[ \bV_{h,0}^n-\bu^n\Big]_{\sigma,\vc n_\sigma} + \frac \mu 3\dv\bV^n\Big[\bV^n_{h,0}-\bu^n\Big]_{\sigma,\vc n_\sigma}\cdot\bn_{\sigma}\Big)\dS,
\end{aligned}
\end{equation}
}
where in the last line $\vc n_\sigma$ is the unit  normal to the face $\sigma$ and $[\cdot]_{\sigma,\vc n_\sigma}$ is the jump over sigma (with respect to $\vc n_\sigma$) defined in Lemma \ref{Lemma6}.

To estimate $I_1$, we use the Cauchy-Schwartz inequality, decompose $\vc V^n-\vc V^n_{h,0}=
\vc V^n-\vc V^n_{h}+ \vc V^n_h-\vc V^n_{h,0}$
and employ estimates (\ref{L1-3})$_{s=2}$, (\ref{N2-2}--\ref{N2-3}) to evaluate
 the norms involving $\nabla(\vc V^n-\vc V^n_{h,0})$, and  decompose $\vc V^n_{h,0}=\vc V^n_{h,0}-\vc V_h^n +\vc V_h^n$
use (\ref{N2-2}--\ref{N2-3}), (\ref{L1-2})$_{s=1}$, (\ref{est0}), the Minkowski inequalityt to estimate the norms involving $\nabla(\vc V^n_{h,0}-\vc u^n)$.
We get
{
$$
|I_1|\le h\;c(M_0,E_0, \|\nabla\vc V,\nabla^2 \bV\|_{L^\infty(0,T; L^{\infty}(\Omega;\Rm^{36}))}).
$$
Since the integral over any face $\sigma\in {\cal E}_{\rm int}$ of the jump of a function from $V_{h,0}(\Omega_h)$ is zero, we may write
$$
I_2=\deltat\sum_{n=1}^m\ \sum_{\sigma\in {\cal E}_{\rm int}} \int_{\sigma} \Big(\mu\bn_{\sigma} \cdot\Big(\nabla\bV^n -(\nabla\bV^n)_\sigma\Big)\cdot\Big[\bu^n- \bV_{h,0}^n\Big]_{\sigma,\vc n_\sigma}
$$
$$
+ \frac \mu 3\Big(\dv\bV^n-(\dv\bV^n)_\sigma\Big)\Big[\bu^n-\bV^n_{h,0}\Big]_{\sigma,\vc n_\sigma}\cdot\bn_{\sigma}\Big)\dS;
$$
whence by using the { first order Taylor formula applied to functions $x\mapsto \nabla \vc V^n(x)$ to evaluate the differences $\nabla\bV^n -(\nabla\bV^n)_\sigma$, $\dv\bV^n-[\dv\bV^n]_\sigma$, }
   and H\"older's inequality,
$$
\begin{aligned}
&|I_2| \le \deltat\, h\; c\,  \|\nabla^2\bV\|_{L^\infty(Q_T;\Rm^{27})} \sum_{n=1}^m\sum_{\sigma\in {\cal E}_{\rm int}} \sqrt{|\sigma|}\sqrt h\;\Big(\frac 1{\sqrt h}\,\Big\|\Big[\bu^n-\bV_{h,0}^n\Big]_{\sigma,\vc n_\sigma}\Big\|_{L^2(\sigma;\Rm^3)}\Big)\\
&\le \deltat\, h\; c\,  \|\nabla^2\bV\|_{L^\infty(Q_T;\Rm^{27})} \sum_{n=1}^m\sum_{\sigma\in {\cal E}_{\rm int}}\Big(|\sigma|h+
\frac 1h\,\Big\|\Big[\bu^n-\bV_{h,0}^n\Big]_{\sigma,\vc n_\sigma}\Big\|_{L^2(\sigma;\Rm^3)}^2\Big).
\end{aligned}
$$
Therefore,
\begin{equation}\label{cR1.1}
|{\cal R}_{1,1}|\le  h\, c(M_0, E_0,\|\vc V, \nabla\bV, \nabla^2\bV\|_{L^\infty(Q_T,\Rm^{39})}),
\end{equation}
where we have employed Lemma \ref{Lemma6}, (\ref{est0}) and (\ref{N2-2}--\ref{N2-3}), (\ref{L1-2}).
}

\vspace{2mm}
\textbf{Step 2:}\textit{  Term  ${\cal T}_2$.}\label{7.4}
Let us now decompose the term  ${\cal T}_2$ as
\begin{align*}
	 & {\cal T}_2={\cal T}_{2,1}+ {\cal R}_{2,1},\\
	& \mbox{with } {\cal T}_{2,1}=\deltat\sum_{n=1}^m\sum_{K\in {\cal T}}\int_Kr^{n-1}\frac{\bV^n-\bV^{n-1}}{\deltat}\cdot (\bV^n_{h,0}-\bu^n)\dx,\quad {\cal R}_{2,1}=\deltat\sum_{n=1}^m\sum_{K\in {\cal T}}{\cal R}_{2,1}^{n,K}, \\
	& \mbox{and }{ {\cal R}_{2,1}^{n,K}=\int_K(r^n-r^{n-1})[\partial_t\vc V]^n\cdot(\bV^n_{h,0}-\bu^n)\dx }+ \int_Kr^{n-1}\Big([\partial_t \bV]^n-\frac{\bV^n-\bV^{n-1}}{\deltat}\Big) \cdot(\bV_{h,0}^n-\bu^n)\dx.
\end{align*}
The remainder ${\cal R}_{2,1}^{n,K}$ can be rewritten as follows
$$
{\cal R}_{2,1}^{n,K}=\int_K\Big[\int_{t_{n-1}}^{t_n}\partial_tr(t,\cdot){\rm d}t\Big][\partial_t\vc V]^n\cdot(\bV^n_{h,0}-\bu^n)\dx + \frac 1\deltat\int_Kr^{n-1}\Big[\int_{t_{n-1}}^{t_n}\int_s^{t_n}\partial^2_t \bV(z,\cdot){\rm d}z{\rm d}s\Big] \cdot(\bV_{h,0}^n-\bu^n)\dx;
$$
{ whence, by the H\"older inequality,
\[
|{\cal R}_{2,1}^{n,K}|\le \deltat\Big[ (\|r\|_{L^\infty(Q_T)}+\|\partial_t r\|_{L^\infty(Q_T)}) (\|\partial_t\bV\|_{L^\infty(Q_T;\Rm^3)}|K|^{5/6}(\|\bu^n\|_{L^6(K)}+ \|\bV^n_{h,0}\|_{L^6(K)})
 \]
 \[
 +\|\partial^2_t \bV^n\|_{L^{6/5}(\Omega;\Rm^3))} (\|\bu^n\|_{L^6(K)}+ \|\bV^n_{h,0}\|_{L^6(K)})\Big].
\]
Consequently, by the same token as in (\ref{R5.2}) or (\ref{R6.4}),
\begin{equation}\label{cR2.1}
|{\cal R}_{2,1}|\le \deltat\, c\Big(M_0, E_0,\overline r,\|(\partial_t r, \bV, \partial_t\bV, \nabla\bV  )\|_{L^\infty(Q_T;\Rm^{16})}, \|\partial^2_t \bV\|_{L^2(0,T; L^{6/5}(\Omega;\Rm^3))}\Big),
\end{equation}
 where we have used the discrete H\"older and Young inequalities,  the estimates (\ref{ddd}), (\ref{N2-2}--\ref{N2-3}) and  the energy bound (\ref{est0}) from Corollary \ref{Corollary1}.
 }

\vspace{2mm}
{\bf Step 2a:} {\it Term ${\cal T}_{2,1}$.} We decompose the term ${\cal T}_{2,1}$ as
\begin{align*}
&{\cal T}_{2,1}={\cal T}_{2,2}+ {\cal R}_{2,2}, \\
&\mbox{with } {\cal T}_{2,2}=\deltat\sum_{n=1}^m\sum_{K\in {\cal T}}\int_Kr_K^{n-1}\frac{\bV^n-\bV^{n-1}}{\deltat}\cdot (\bV^{n}_{h,0}-\bu^{n})\dx, \;
 {\cal R}_{2,2}=\deltat\sum_{n=1}^m\sum_{K\in {\cal T}}{\cal R}_{2,2}^{n,K}, \\
&\mbox{and }{\cal R}_{2,2}^{n,K}=\int_K(r^{n-1}-r_K^{n-1})\frac{\bV^n-\bV^{n-1}}{\deltat}\cdot(\bV_{h,0}^n-\bu^{n})\dx;
\end{align*}
therefore,
\[
|{\cal R}_{2,2}^{n}|= |\sum_{K\in {\cal T}}{\cal R}_{2,2}^{n,K}| \le h \, c\|\nabla r\|_{L^\infty(Q_T;\Rm^3)}\|\partial_t\bV\|_{L^\infty(Q_T;\Rm^{3})}\|\bu^n-\bV^n_{h,0}\|_{L^6(\Omega;\Rm^3)}.
\]
Consequently, by virtue of formula (\ref{est1}) for $\vc u^n$  and estimates (\ref{ddd}), (\ref{N2-2}--\ref{N2-3}),
\begin{equation}\label{cR2.2}
|{\cal R}_{2,2}|\le h \, c(M_0, E_0, \|(\nabla r, \bV,\partial_t\bV,\nabla\bV)\|_{L^\infty(Q_T;\Rm^{18})}).
\end{equation}

{
\vspace{2mm}
{\bf Step 2b:} {\it Term ${\cal T}_{2,2}$.} We decompose the term ${\cal T}_{2,2}$ as
\begin{align*}
&{\cal T}_{2,2}={\cal T}_{2,3}+ {\cal R}_{2,3}, \\
&\mbox{with } {\cal T}_{2,3}=\deltat\sum_{n=1}^m\sum_{K\in {\cal T}}\int_Kr_K^{n-1}\frac{\bV^n_{h,0,K}-\bV^{n-1}_{h,0,K}}{\deltat}\cdot (\bV^{n}_{h,0}-\bu^{n})\dx, \;
 {\cal R}_{2,3}=\deltat\sum_{n=1}^m\sum_{K\in {\cal T}}{\cal R}_{2,3}^{n,K}, \\
&\mbox{and }{\cal R}_{2,3}^{n,K}=\int_K r_K^{n-1}\Big(\frac{\bV^n-\bV^{n-1}}{\deltat}-\Big[\frac{\bV^n-\bV^{n-1}}{\deltat}\Big]_h\Big)
\cdot(\bV_{h,0}^n-\bu^{n})\dx\\ &
+\int_K r_K^{n-1}\Big(\Big[\frac{\bV^n-\bV^{n-1}}{\deltat}\Big]_h-\Big[\frac{\bV^n-\bV^{n-1}}{\deltat}\Big]_{h,K}\Big)
\cdot(\bV_{h,0}^n-\bu^{n})\dx\\ &
+
\int_K r_K^{n-1}\Big(\Big[\frac{\bV^n-\bV^{n-1}}{\deltat}\Big]_{h,K}-\Big[\frac{\bV^n-\bV^{n-1}}{\deltat}\Big]_{h,0,K}\Big)
\cdot(\bV_{h,0}^n-\bu_n)\dx { =I_1^K+I_2^K}+ I^K_3.
\end{align*}
{ We calculate carefully
$$
|I^K_3|= \frac 1 \deltat r_K^{n-1}\int_K\Big\{
\int_{t_{n-1}}^{t_n}\Big[[\partial_t\bV(z)]_h- [\partial_t\bV(z)]_{h,0}\Big]_K\cdot
(\bV_{h,0}^n-\bu^{n}){\rm d}z\Big\}\dx
$$
$$
\le
\frac 1 \deltat r_K^{n-1}\int_{t_{n-1}}^{t_n}\Big\|\Big[[\partial_t\bV(z)]_h- [\partial_t\bV(z)]_{h,0}\Big]_K\Big\|_{L^{6/5}(K;\Rm^3)}
\|\bV_{h,0}^n-\bu^{n}\|_{L^6(K;\Rm^3)}{\rm d}z.
$$
Summing over polyhedra $K\in {\cal T}$
we get simply by using the discrete Sobolev inequality
$$
\sum_{K\in {\cal T}}|I_3^K|\le \frac 1 \deltat r_K^{n-1} \int_{t_{n-1}}^{t_n}\Big\{
\Big(\sum_{K\in {\cal T}}\|\bV^n_{h,0}-\bu^n\|^6_{L^6(K;\Rm^3)}\Big)^{1/6}\Big(\sum_{K\in {\cal T}}\|[\partial_t\bV(z)]_h-[\partial_t\bV(z)]_{h,0}\Big\|^{6/5}_{L^{6/5}(K;\Rm^3)}\Big)^{5/6}
\Big\}
{\rm d}z
$$
$$
\le
\frac 1 \deltat r_K^{n-1} \int_{t_{n-1}}^{t_n}\|\bV^n_{h,0}-\bu^n\|_{L^6(\Omega_h;\Rm^3)}
\|[\partial_t\bV(z)]_h-[\partial_t\bV(z)]_{h,0}\Big\|_{L^{6/5}(\Omega_h;\Rm^3)}{\rm d} z
$$
$$
\le \frac {h^{5/6}} {\deltat} \int_{t_{n-1}}^{t_n}\|\bV^n_{h,0}-\bu^n\|_{L^6(\Omega_h;\Rm^3)}
\|\partial_t\bV(z)\|_{L^{\infty}(\Omega_h;\Rm^3)} {\rm d} z,
$$
where we have used estimate (\ref{N2-4}) to obtain the last line.

As far as the term $I_2^K$ is concerned, we write
$$
|I_2^K|=\frac 1 \deltat r_K^{n-1}\Big|\int_K \Big(\Big[\int_{t_{n-1}}^{t_n}\partial_t\bV(z){\rm d}z\Big]_h- \Big[\int_{t_{n-1}}^{t_n}\partial_t\bV(z){\rm d}z\Big]_{h,K}\Big)
\cdot(\bu^{n}-\bV_{h,0}^n)\dx\Big|
$$
$$
\le \frac h \deltat r_K^{n-1}\int_{t_{n-1}}^{t_n}\Big\|\Grad\Big[\partial_t\bV(z)\Big]_h\Big\|_{L^{6/5}(K;\Rm^3)}
\|\bu^n-\bV^n_{h,0}\|_{L^6(K;\Rm^3)},
$$
where we have used the Fubini theorem, H\"older's inequality and (\ref{L2-1}), (\ref{L1-3})$_{s=1}$.
Further, employing the Sobolev inequality on the Crouzeix-Raviart space $V_h(\Omega_h)$ (\ref{sob1}), the H\"older inequality and estimate (\ref{L1-3})$_{s=1}$, we get
$$
\sum_{K\in {\cal T}}|I_2^K|\le \frac h \deltat r_K^{n-1} \|\bu^n-\bV^n_{h,0}\|_{L^6(\Omega_h;\Rm^3)}\int_{t_{n-1}}^{t_n}\Big\|\Grad\partial_t\bV(z)\Big\|_{L^{6/5}(\Omega_h;\Rm^3)}
{\rm d}z.
$$

We reserve the similar treatment to the term $I_1^K$. Resuming these calculations and summing over $n$ from $1$ to $m$ we get by using Corollary  \ref{Corollary1} and estimates (\ref{N2-2}--\ref{N2-3}), (\ref{ddd}),
%
%
\begin{equation}\label{cR2.3}
|{\cal R}_{2,3}|\le h^{5/6} \, c(M_0, E_0, \|(r, \bV,\nabla\bV, \partial_t\bV)\|_{L^\infty(Q_T;\Rm^{16})},\|\partial_t\nabla\bV\|_{L^2(0,T;L^{6/5}(\Omega;\Rm^{9}))}).
\end{equation}
}
}
{ \bf Step 2c:} {\it Term ${\cal T}_{2,3}$.}
{ We rewrite this term in the form
\begin{equation}\label{cT2}\begin{aligned}
& {\cal T}_{2,3}={\cal T}_{2,4}+ {\cal R}_{2,4},\; {\cal R}_{2,4}=\deltat\sum_{n=1}^m\sum_{K\in {\cal T}}{\cal R}_{2,4}^{n,K}, \\
&\mbox{with } {\cal T}_{2,4}=\deltat\sum_{n=1}^m\sum_{K\in {\cal T}} \int_Kr_K^{n-1} \frac {\bV^n_{h,0,K}-\bV^{n-1}_{h,0,K}} {\deltat}\cdot (\bu^{n}_K-\bV^{n}_{h,0,K})\dx,
\\ &\mbox{and }{\cal R}_{2,4}^{n,K}=\int_Kr_K^{n-1}\frac{\bV_{h,0,K}^n-\bV_{h,0,K}^{n-1}}{\deltat}\cdot\Big((\bu^{n}-\bu^n_K) -(\bV_{h,0}^n-\bV_{h,0,K}^n)\Big)\dx.
\end{aligned}
\end{equation}
First, we estimate the $L^\infty$ norm of $\frac {\bV^n_{h,0,K}-\bV^{n-1}_{h,0,K}}\deltat$ as in (\ref{R2.1}).
Next, we decompose
$$
\bV^n_{h,0}-\bV^n_{h,0,K}=\bV^n_{h,0}-\bV^n_{h}+\bV^n_h-\bV^n_{h,K}+[\bV^n_{h}-\bV^n_{h,0}]_K,
$$
{and use (\ref{L2-1})$_{p=2}$ to estimate $\bu^n-\bu^n_K$, (\ref{L2-1})$_{p=\infty}$, (\ref{L1-3})$_{s=1}$ to estimate
$\bV^n_h-\bV^n_{h,K}$ and (\ref{N2-2}--\ref{N2-3}) to evaluate $\|[\bV^n_{h}-\bV^n_{h,0}]_K\|_{L^\infty(K;\Rm^3)}\le
\|\bV^n_{h}-\bV^n_{h,0}\|_{L^\infty(K;\Rm^3)}$. Thanks to the H\"older inequality and (\ref{est0}) we finally deduce
\begin{equation}\label{cR2.4}
|{\cal R}_{2,4}|\le h \; c(M_0,E_0,\overline r, \|(\bV, \partial_t\bV, \nabla\bV)\|_{L^\infty(Q_T;\Rm^{15})}).
\end{equation}
}
}

\vspace{2mm}
\textbf{Step 3:} \textit{ Term  ${\cal T}_3$.}
Let us first decompose ${\cal T}_3$ as
\begin{align*}
& {\cal T}_3={\cal T}_{3,1} + {\cal R}_{3,1}, \\
&\mbox{with } {\cal T}_{3,1}=\deltat\sum_{n=1}^m\sum_{K\in{\cal T}}\int_K r_K^n\bV_{h,0,K}^n\cdot\nabla\bV^n\cdot (\bV_{h,0,K}^n - \bu_K^n)\dx,
 \quad {\cal R}_{3,1}=\deltat\sum_{n=1}^m \sum_{K\in{\cal T}}{\cal R}_{3,1}^{n,K},
\\ &\mbox{and }{\cal R}_{3,1}^{n, K}=\int_K(r^n-r_K^n)\bV^n\cdot\nabla\bV^n\cdot(\bV^n_{h,0}-\bu^n)\dx
+ \int_K r_K^n(\bV^n-\bV^n_{h,0})\cdot\nabla\bV^n\cdot(\bV^n_{h,0}-\bu^n)\dx \\
& \phantom{\mbox{and }{\cal R}_{3,1}^{n, K}=} +\int_K r_K^n(\bV_{h,0}^n-\bV^n_{h,0,K})\cdot\nabla\bV^n\cdot(\bV^n_{h,0}-\bu^n)\dx \\
& \phantom{\mbox{and }{\cal R}_{3,1}^{n, K}=}
+\int_K r_K^n\bV^n_{h,0,K}\cdot\nabla\bV^n\cdot\Big(\bV^n_{h,0}-\bV^n_{h,0,K} - (\bu^n-\bu^n_K)\Big)\dx.
\end{align*}
We have
$$
\|r^n-r^n_K\|_{L^\infty(K)}\aleq h \|\nabla r^n\|_{L^\infty(K)},
$$
by the Taylor formula,
$$
\|\bV^n-\bV_{h,0}^n\|_{L^\infty(K;\Rm^3)}\aleq h \|\nabla \bV^n\|_{L^\infty(K;\Rm^9)},
$$
by virtue of (\ref{L1-2})$_{s=1}$ and (\ref{N2-2}--\ref{N2-3}),
$$
\|\bV_{h,0}^n-\bV_{h,0,K}^n\|_{L^\infty(K;\Rm^3)}\le \|\bV_{h,0}^n-\bV_{h,}^n\|_{L^\infty(K;\Rm^3)}+
\|\bV_{h}^n-\bV_{h,K}^n\|_{L^\infty(K;\Rm^3)}
$$
$$
+ \|[\bV_{h}^n-\bV_{h,0}^n]_K\|_{L^\infty(K;\Rm^3)}
\aleq h \|\nabla \bV^n\|_{L^\infty(K;\Rm^9)}
$$
by virtue of  (\ref{L2-1}), (\ref{L1-2})$_{s=1}$ (\ref{L1-3})$_{s=1}$ and (\ref{N2-2}--\ref{N2-3}),
$$
\|\bu^n-\bu^n_K\|_{L^\infty(K;\Rm^3)}\aleq h \|\nabla \bu^n\|_{L^\infty(K;\Rm^9)}.
$$
Consequently by employing several times the H\"older inequality (for integrals over $K$) and the discrete
H\"older inequality (for the sums over $K\in{\cal T}$), and using estimate (\ref{est0}), we arrive at

\begin{equation}\label{cR3.1}
|{\cal R}_{3,1}|\le h\; c(M_0, E_0,\overline r, \|(\nabla r, \bV, \nabla\bV)\|_{L^\infty(Q_T;\Rm^{15})}).
\end{equation}

\vspace{1mm}\noindent
Now we shall deal wit term ${\cal T}_{3,1}$. Integrating by parts, we get:

\begin{align*}
	\int_K r_K^n\bV_{h,0,K}^n\cdot\nabla\bV^n\cdot (\bV_{h,0,K}^n-\bu_K^n)\dx &=\sum_{\sigma\in {\cal E}(K)}{ |\sigma|} r_K^n  [\bV_{h,0,K}^n\cdot{\vc n}_{\sigma,K}]\bV_\sigma^n\cdot(\bV_{h,0,K}^n-\bu_K^n)
	\\
	&=\sum_{\sigma\in {\cal E}(K)}|\sigma| r_K^n [\bV_{h,0,K}^n\cdot{\vc n}_{\sigma,K}](\bV_\sigma^n-\bV^n_{h,K})\cdot(\bV_{h,K}^n-\bu_K^n),
\end{align*}
thanks to the the fact that $\sum_{\sigma\in {\cal E}(K)}\int_\sigma \bV_{h,K}^n\cdot{\vc n}_{\sigma,K}{\rm d} S=0$.

{ Next we write
\begin{align*}
& {\cal T}_{3,1}= {\cal T}_{3,2}+ {\cal R}_{3,2},  \quad {\cal R}_{3,2}=\deltat\sum_{n=1}^m {\cal R}_{3,2}^{n},
\end{align*}
\begin{align}
&
{\cal T}_{3,2}=\deltat\sum_{n=1}^m\sum_{K\in{\cal T}}\sum_{\sigma\in {\cal E}(K)}|\sigma| \hat r_\sigma^{n,{\rm up}}
[\hat \bV_{h,0,\sigma}^{n,{\rm up}}\cdot{\vc n}_{\sigma,K}](\bV_\sigma^n-\bV^n_{h,K})\cdot(\hat\bV_{h,0,\sigma}^{n,{\rm up}} - \hat\bu_\sigma^{n,{\rm up}}),
\end{align}
\begin{align*}
&  \mbox{and }{\cal R}_{3,2}^{n}=\sum_{K\in {\cal T}}\sum_{\sigma\in {\cal E}(K)}|\sigma|(r_K^n-\hat r^{n,{\rm up}}_\sigma) [\bV_{h,0,K}^n\cdot{\vc n}_{\sigma,K}](\bV_\sigma^n-\bV^n_{h,K})\cdot(\bV^n_{h,0,K}-\bu_K^n)
\\
&+
\sum_{K\in {\cal T}}\sum_{\sigma\in {\cal E}(K)}|\sigma|\hat r^{n,{\rm up}}_\sigma \Big[\Big(\bV_{h,0,K}^n-\hat \bV_{h,0,\sigma}^{n,{\rm up}}\Big)\cdot{\vc n}_{\sigma,K}\Big](\bV_\sigma^n-\bV^n_{h,K})\cdot(\bV^n_{h,K}-\bu_K^n)\\
&
+\sum_{K\in {\cal T}}\sum_{\sigma\in {\cal E}(K)}|\sigma|\hat r^{n,{\rm up}}_\sigma [\hat \bV_{h,0,\sigma}^{n,{\rm up}}\cdot{\vc n}_{\sigma,K}](\bV_\sigma^n-\bV^n_{h,K})
\cdot\Big((\bV^n_{h,0,K}-\hat\bV_{h,0,\sigma}^{n,{\rm up}})- (\bu_K^n -\hat\bu_{h,\sigma}^{n,{\rm up}}) \Big).
\end{align*}
We may write
$$
\vc V_\sigma^n-{\vc V}_{h,0,K}^n=\vc V_\sigma^n-{\vc V}^n+ \vc V^n-{\vc V}_{h}^n+  \vc V_h^n-{\vc V}_{h,K}^n  + [\vc V_h^n-{\vc V}^n_{h,0}]_K,$$
and use several times the { Taylor formula  along with (\ref{L1-2})$_{s=1}$, (\ref{L2-1}), (\ref{L1-3})$_{s=1}$,
(\ref{N2-2}--\ref{N2-3})(in order to estimate $r_K^n-\hat r_\sigma^{n,{\rm up}}$,
 $\vc V_\sigma^n-{\vc V}_{h,0,K}^n$, $\vc V_{h,K}^n-\hat{\vc V}_{h,\sigma}^{n,{\rm up}}$)}
 to get the bound
$$
|{\cal R}_{3,2}^n|\le h\, c \|r\|_{W^{1,\infty}(\Omega)}\Big(1+ \|\bV\|_{W^{1,\infty}(Q_T;\Rm^3)}\Big)^3
\sum_{K\in {\cal T}} h|\sigma||\bu^n_K|
$$
$$
+ c \|r\|_{W^{1,\infty}(\Omega)}\Big(1+ \|\bV\|_{W^{1,\infty}(Q_T;\Rm^3)}\Big)^2\sum_{K\in{\cal T}}
\sum_{\sigma\in {\cal E}(K)} h|\sigma| |\bu_K^n-\bu_\sigma^n|.
$$
We have by the H\"older inequality
$$
\sum_{K\in {\cal T}} h|\sigma||\bu^n_K|\le c\Big(\sum_{\sigma\in {\cal T}} h|\sigma||\bu^n_K|^6\Big)^{1/6}
\le c\Big[\Big(\sum_{K\in {\cal T}}\|\bu^n-\bu_K^n\|^6_{L^6(K;\Rm^3)}\Big)^{1/6}
$$
$$
+\Big(\sum_{K\in {\cal T}}\|\bu^n\|^6_{L^6(K;\Rm^3)}\Big)^{1/6}\Big]
\le c\Big(\sum_ {K\in {\cal T}}\|\nabla\bu_n\|^2_{L^2(K;\Rm^9)}\Big)^{1/2},
$$
$$
\sum_{K\in{\cal T}}
\sum_{\sigma\in {\cal E}(K)} h|\sigma| |\bu_K^n-\bu_\sigma^n|\le c\Big[\Big(\sum_{K\in {\cal T}}\|\bu^n-\bu_K^n\|^2_{L^2(K;\Rm^3)}\Big)^{1/2}
$$
$$
+ \Big(\sum_{K\in {\cal T}}\sum_{\sigma\in {\cal E}(K)}\|\bu^n-\bu_\sigma^n\|^2_{L^2(K;\Rm^3)}\Big)^{1/2}\Big]
\le h\,c\Big(\sum_{K\in {\cal T}} \|\nabla\bu_n\|^2_{L^2(K;\Rm^9)}\Big)^{1/2},
$$
where we have used (\ref{L2-3})$_{p=2}$, (\ref{L2-1}--\ref{L2-2})$_{p=2}$.
Consequently, we may use (\ref{est0}) to conclude
\begin{equation}\label{cR3.2}
|{\cal R}_{3,2}|\le h\, c\Big(M_0, E_0, \overline r, \|\nabla r, \bV, \nabla\bV\|_{L^{\infty}(Q_T;\Rm^{15})}\Big).
\end{equation}
}
Finally, we replace in ${\cal T}_{3,2}$ $\bV_\sigma^n-\bV^n_{h,K}$ by $\bV_{h,0,\sigma}^n-\bV^n_{h,0,K}$. We get
\begin{align*}
& {\cal T}_{3,2}= {\cal T}_{3,3}+ {\cal R}_{3,3},  \quad {\cal R}_{3,3}=\deltat\sum_{n=1}^m {\cal R}_{3,3}^{n},
\end{align*}
\begin{align}
&
{\cal T}_{3,3}=\deltat\sum_{n=1}^m\sum_{K\in{\cal T}}\sum_{\sigma\in {\cal E}(K)}|\sigma| \hat r_\sigma^{n,{\rm up}}
[\hat \bV_{h,0,\sigma}^{n,{\rm up}}\cdot{\vc n}_{\sigma,K}](\bV_{h,0,\sigma}^n-\bV^n_{h,0,K})\cdot(\hat\bV_{h,0,\sigma}^{n,{\rm up}}- \hat\bu_\sigma^{n,{\rm up}}), \label{cT3}
\end{align}
and
\begin{align*}
{\cal R}_{3,3}^{n}=\sum_{K\in {\cal T}}\sum_{\sigma\in {\cal E}(K)}|\sigma|(r_K^n-\hat r^{n,{\rm up}}_\sigma) \bV_{h,0,K}^n\cdot{\vc n}_{\sigma,K}\Big([\bV^n-\bV_{h,0}]_\sigma^n-[\bV^n_{h}-\bV^n_{h,0}]_K\Big)\cdot(\hat\bV_{h,0,\sigma}^{n,{\rm up}}- \hat\bu_\sigma^{n,{\rm up}}),
\end{align*}
committing error
\begin{equation}\label{cR3.3}
|{\cal R}_{3,3}^{n}|= \le h\, c\Big(M_0, E_0, \overline r, \|\nabla r, \bV, \nabla\bV\|_{L^{\infty}(Q_T;\Rm^{15})}\Big),
\end{equation}
as in the previous step.

\vspace{2mm}\noindent
\textbf{Step 4:}\textit{ Terms  ${\cal T}_4$ }
We write
$$
{\cal T}_4= {\cal T}_{4,1}+{\cal R}_{4,1},\;\;{\cal T}_{4,1}=-\int_{\Omega_h}\nabla p(r^n)\cdot\vc V^n{\rm d}x,
$$
$$
{\cal R}_{4,1}=\int_{\Omega_h}\nabla p(r^n)\cdot(\vc V^n-\bV^n_{h,0})\dx;
$$
whence
\bFormula{cR4.1}
|{\cal R}_{4,1}|\le h c(\overline r, |p'|_{C[\underline r,\overline r]}, \|\nabla r\|_{L^\infty(Q_T;\Rm^3)}),
\eF
by virtue of (\ref{L1-2})$_{s=1}$, (\ref{N2-2}--\ref{N2-3}).

Next, employing the integration by parts
$$
{\cal T}_{4,2}={\cal T}_{4,2}+{\cal R}_{4,2},\quad
{\cal T}_{4,2}=\int_{\Omega_h} p(r^n)\dv\vc V^n\dx,
$$
$$
{\cal R}_{4,2}=-\sum_{K\in {\cal T}}\sum_{\sigma\in {\cal E}(K), \sigma\in\partial\Omega_h}\int_\sigma
p(r^n)\vc V^n\cdot\vc n_{\sigma,K}{\rm d}S= -\sum_{K\in {\cal T}}\sum_{\sigma\in {\cal E}(K), \sigma\in\partial\Omega_h}\int_\sigma
p(r^n)\Big(\vc V^n-\bV^n_{h,0,\sigma}\Big)\cdot\vc n_{\sigma,K}{\rm d}S.
$$
Writing
$$
\vc V^n-\bV^n_{h,0,\sigma}=\vc V^n-\bV^n_{h}+\vc V_h^n-\bV^n_{h,\sigma}+[\bV_{h}^n-\bV^n_{h,0}]_\sigma,
$$
we deduce by using (\ref{L1-2})$_{s=1}$, (\ref{L1-3})$_{s=1}$, (\ref{L2-2})$_{p=\infty}$, (\ref{N2-2}),
(\ref{N2-3}),
$$
\|\vc V^n-\bV^n_{h,0,\sigma}\|_{L^\infty(K;\Rm^3)}\aleq h\|\nabla \bV^n\|_{L^\infty(K;\Rm^3)},\;\sigma\in K.
$$
Now, we employ the fact that
$$
\sum_{K\in {\cal T}}\sum_{\sigma\in {\cal E}(K), \sigma\in\partial\Omega_h}\int_\sigma{\rm d}S\approx 1;
$$
whence
\bFormula{cR4.2}
|{\cal R}_{4,2}|\le h c(\overline r, |p|_{C[\underline r,\overline r]}, \|\nabla \vc V\|_{L^\infty(Q_T;\Rm^9)})
\eF

Finally,
\bFormula{cT4}
{\cal T}_{4,2}={\cal T}_{4,3}+{\cal R}_{4,3},\quad
{\cal T}_{4,3}=\int_{\Omega_h} p(\hat r^n)\dv\vc V^n\dx,\; \quad {\cal R}_{4,3}=\int_{\Omega_h}( p( r^n)-p(\hat r^n))\dv\vc V^n\dx;
\eF
whence
\bFormula{cR4.3}
|{\cal R}_{4,3}|\le h c(|p'|_{C[\underline r,\overline r]}, \|(\nabla r,\nabla\vc V)\|_{L^\infty(Q_T;\Rm^{12})}).
\eF

\vspace{2mm}\noindent
\textbf{Step 5:} \textit{Term ${\cal T}_6$} \label{7.6}
We decompose ${\cal T}_6$ as

\begin{equation}\label{cT6}
\begin{aligned}
	& {\cal T}_6= {\cal T}_{6,1}+ {\cal R}_{6,1},
\mbox{with }{\cal T}_{6,1}=-\deltat \sum_{n=1}^m \sum_{K\in{\cal T}}\int_K  p'(\hat r^n)\bu^n\cdot\nabla r^n\dx,\\
&
{\cal R}_{6,1}= \deltat \sum_{n=1}^m \sum_{K\in{\cal T}}\int_K (p'(\hat r^n)-p'(r^n))\cdot\bu^n\cdot\nabla r^n\dx;
\end{aligned}
\end{equation}
Consequently, by the Taylor formula, H\"older inequality and estimate (\ref{est1}),

\begin{equation}\label{cR6.1}
	|{\cal R}_{6,1}|\le h\, c(M_0, E_0,\underline r,\overline r, |p'|_{C^1([\underline r,\overline r])}, \|\nabla r\|_{L^\infty(Q_T;\Rm^3)}).
\end{equation}

Gathering the formulae (\ref{cT1}), (\ref{cT2}), (\ref{cT3}), (\ref{cT4}), (\ref{cT6}) and estimates for the residual terms (\ref{cR1.1}), (\ref{cR2.1}--\ref{cR2.4}), (\ref{cR3.1}--\ref{cR3.3}), (\ref{cR4.1}), (\ref{cR4.2}), (\ref{cR4.3}), (\ref{cR6.1}) concludes the proof of Lemma \ref{strongentropy}.
\qed

\section{A Gronwall inequality}

{ In this Section we put together the relative energy inequality (\ref{relativeenergy-}) and the identity (\ref{relativeenergy-**}) derived in the previous section. The final inequality resulting from this manipulation is formulated in the following lemma.}
\begin{Lemma}\label{Gronwall}
Let  $(\vr^n,\bu^n)$ be a solution of the discrete problem (\ref{num1}--\ref{num3}) with the pressure satisfying (\ref{i4}), where $\gamma\ge 3/2$.
Then  there exists a positive number
\[
c=c\Big( M_0,E_0,\underline r, \overline r, |p'|_{{C^1}[\underline r,\overline r]}, \|(\partial_t r, \nabla r, \bV,
\partial_t\bV, \nabla\bV, \nabla^2\bV)\|_{L^\infty(Q_T;\Rm^{45})},
\]
\[
 \|\partial_t^2r\|_{L^1(0,T;L^{\gamma'}(\Omega))}, \|\partial_t\nabla r\|_{L^2(0,T;L^{6\gamma/{5\gamma-6}}(\Omega;\Rm^3))}, \|\partial_t^2\bV,\partial_t\nabla\bV\|_{L^2(0,T;L^{6/5}(\Omega;\Rm^{12}))}
\Big),
\]
such that for all $m=1,\ldots,N,$ there holds:
\[{\cal E}(\vr^m,\bu^m|\hat r^m,\hat\bV_{h,0}^m)
{ +\deltat\frac \mu 2\sum_{n=1}^m\sum_{{K}\in {\cal T}}\int_K|\Grad(u^n-\vc V^n_{h,0})|^2{\rm d x}}
\]
\[
\le c\Big[h^a+\sqrt{\deltat} + {\cal E}(\vr^0,\bu^0|\hat r(0),\hat\bV_{h,0}(0))\Big] + c\,\deltat\sum_{n=1}^m {\cal E}(\vr^n,\bu^n|\hat r^n,\hat\bV_{h,0}^n),
\]
with any couple $(r,\vc V)$ belonging to (\ref{XT}) and satisfying the continuity equation
(\ref{i1}) on $(0,T)\times\R^3$ and momentum equation (\ref{i2}) with boundary conditions (\ref{i6}) on $(0,T)\times\Omega$ in the classical sense,
where $a$ is defined in (\ref{A1}) and ${\cal E}$ is given in (\ref{est4}).
\end{Lemma}

\begin{proof}
We observe that
$$
S_6-{\cal S}_6= \deltat \sum_{n=1}^m\int_{\Omega_h} p'(\hat r^n)\frac {\hat r^n-\vr^n}{\hat r^n}\vc V^n\cdot\nabla r^n\dx + \deltat \sum_{n=1}^m\int_{\Omega_h} p'(\hat r^n)\frac {\hat r^n-\vr^n}{\hat r^n}(\vc u^n-\vc V^n)\cdot\nabla r^n\dx.
$$
Gathering the formulae (\ref{relativeenergy-}) and (\ref{relativeenergy-*}), one gets
\begin{equation}\label{relativeenergy-1}
{\cal E}(\vr^m,\bu^m\Big| \hat r^m, \hat \bV_{h,0}^m)- {\cal E}(\vr^0,\bu^0\Big|\hat r(0),\hat\bV_{h,0}(0)) + \mu \deltat\sum_{n=1}^m\sum_{K\in {\cal T}}\Big|\nabla(\bu^n-\bV^n_{0,h})\Big|^2_{L^2(K;\Rm^3)}\le\sum_{i=1}^4{\cal P}_i+ {\cal Q},
\end{equation}
where
\begin{align*}
&{\cal P}_1=
\deltat\sum_{n=1}^m\sum_{K\in{\cal T}}|K|(\vr_K^{n-1}-r^{n-1}_K)\frac{{\bV}_{h,0,K}^{n}-{\bV}_{h,0,K}^{n-1}}{\deltat}\cdot \Big({\bV}_{h,0,K}^{n}   - \bu_K^{n}\Big),
\\
&{ {\cal P}_2=\deltat\sum_{n=1}^m\sum_{K\in{\cal T}}\sum_{\sigma=K|L\in {\cal E}_K}|\sigma|\Big(\vr_\sigma^{n,{\rm up}}-\hat r_\sigma^{n,{\rm up}}\Big)\Big({\hat\bV}^{n,{\rm up}}_{h,0,\sigma}-
{\hat\bu}^{n,{\rm up}}_\sigma\Big)\cdot\Big(\bV^n_{h,0,\sigma}-\bV^n_{h,0,K}\Big)  \bV_{h,0,\sigma}^{n, {\rm up}}\cdot\bn_{\sigma,K},}
\\
&{\cal P}_3=
-\deltat \sum_{n=1}^m\int_{\Omega_h}\Big(p(\vr^n)-p'(\hat r^n)(\vr^n-\hat r^n)-p(\hat r^n)\Big){\rm div}\vc V^n,
\\
& {\cal P}_4=\deltat \sum_{n=1}^m\sum_{K\in {\cal T}}\int_{K} p'(\hat r^n)\frac {\hat r^n-\vr^n}{\hat r^n}(\vc u^n-\vc V^n)\cdot\nabla r^n\dx,
\\
&{ {\cal Q}=
{\cal R}^m_{h,\deltat} +R^m_{h,\deltat} +{ G}^m. }
\end{align*}

Now, we estimate conveniently the terms ${\cal P}_i$, $i=1,\ldots, 4$ in four steps.

\vspace{2mm}

{\bf Step 1:} {\it Term ${\cal P}_1$.} { We estimate the $L^\infty$ norm of $\frac{{\bV}_{h,0,K}^{n}-{\bV}_{h,0,K}^{n-1}}{\deltat}$ by
$L^\infty$ norm of $\partial_t \vc V$ in the same manner as in (\ref{R2.1}).
According to Lemma \ref{LL1},
$
|\vr-r|^\gamma 1_{R_+\setminus[\underline r/2,2\overline r]}(\vr)\le c(p) E^p(\vr|r),
$
with any $p\ge 1$; in particular,
\begin{equation}\label{aaa}
|\vr-r|^{6/5} 1_{R_+\setminus[\underline r/2,2\overline r]}(\vr)\le c E(\vr|r)
\end{equation}
provided $\gamma\ge 6/5$.

We get by using the H\"older inequality,
$$
\Big|\sum_{K\in{\cal T}}|K|(\vr_K^{n-1}-r^{n-1}_K)\frac{{\bV}_{h,0,K}^{n}-{\bV}_{h,0,K}^{n-1}}{\deltat}\cdot \Big({\bV}_{h,K}^{n}   - \bu_K^{n}\Big)\Big|\le c\|\partial_t\bV \|_{L^\infty(Q_T;\Rm^3)}\times
$$
$$
\Big[\Big(\sum_{K\in{\cal T}}|K||\vr^{n-1}_K-r^{n-1}_K|^2 1_{[\underline r/2,2\overline r]}(\vr_K)\Big)^{1/2}
+
\Big(\sum_{K\in{\cal T}}|K||\vr^{n-1}_K-r^{n-1}_K|^{6/5} 1_{R_+\setminus [\underline r/2,2\overline r]}(\vr_K)\Big)^{5/6}\Big]
\times
$$
$$
\Big(\sum_{K\in{\cal T}}|K| \Big|{\bV}_{h,0,K}^{n}   - \bu_K^{n}\Big|^6\Big)^{1/6}
\le c(\|(\partial_t\bV)\|_{L^\infty(Q_T;\Rm^{3})})\Big({\cal E}^{1/2}(\vr^{n-1},\hat\bu^{n-1}|\hat r^{n-1},\hat\bV_{h,0}^{n-1})
$$
$$
+
{\cal E}^{5/6}(\vr^{n-1},\hat\bu^{n-1}|\hat r^{n-1},\hat \bV_{h,0}^{n-1})\Big)\,\Big(\sum_{K\in {\cal T}}\| {\bV}_{h,0,K}^{n}   - \bu_K^{n}\|_{L^6(K;\Rm^3)}^6\Big)^{1/6},
$$
where we have used (\ref{aaa}) and estimate (\ref{est4}) to obtain the last line.
 Now,  we write ${\bV}_{h,0,K}^{n}   - \bu_K^{n}= ([{\bV}_{h,0}^{n}   - \bu^{n}]_K- ({\bV}_{h,0}^{n}   - \bu^{n}))
+ ({\bV}_{h,0}^{n}   - \bu^{n})$ and use the Minkowski inequality together with formulas (\ref{L2-3}), (\ref{sob1})
to get
$$
\Big(\sum_{K\in {\cal T}}\| {\bV}_{h,0,K}^{n}   - \bu_K^{n}\|_{L^6(K;\Rm^3)}^6\Big)^{1/6}\le
\Big(\sum_{K\in {\cal T}}\| \nabla({\bV}_{h,0}^{n}   - \bu^{n})\|^2_{L^2(K;\Rm^3)}\Big)^{1/2}.
$$
Finally, employing
Young's inequality, and estimate (\ref{est4}), we arrive at
\begin{multline}\label{cP1}
|{\cal P}_1|  \le \;  c({ \delta},M_0,E_0,\underline r,\overline r,\|(\bV,\nabla\bV,\partial_t\bV)\|_{L^\infty(Q_T,\Rm^{15})})
\\
\times \Big(\deltat {\cal E}(\vr^0,\hat \bu^0|\hat r^0,\hat\bV_{h,0}^0)+ \deltat\sum_{n=1}^m{\cal E}(\vr^n,\hat\bu^n|\hat r^n,\hat\bV_{h,0}^n)\Big)
+ \delta \deltat\sum_{n=1}^m\sum_{K\in {\cal T}}\| \nabla({\bV}_{h,0}^{n}   - \bu^{n})\|^2_{L^2(K;\Rm^3)},
\end{multline}
with any $\delta>0$.
}

{\bf Step 2:} {\it Term ${\cal P}_2$.} We rewrite $\bV^n_{h_0,\sigma}- \bV^n_{h_0,K}= \bV^n_{h,\sigma}- \bV^n_{h,K}+
[\bV^n_{h,0}- \bV^n_{h}]_\sigma + [\bV^n_{h,0}- \bV^n_{h}]_K$ and estimate the $L^\infty$ norm of this expression by
$h\|\nabla\bV\|_{L^\infty(Q_T;\Rm^9)}$ by virtue of (\ref{N2-2}--\ref{N2-3}), (\ref{L2-1}--\ref{L2-2}), (\ref{L1-3})$_{s=1}$.
{Now we write ${\cal P}_2=\deltat \sum_{n=1}^m{\cal P}_2^n$ where Lemma \ref{LL1} and the H\"older inequality yield, similarly as in the previous step,
\begin{equation*}
\begin{aligned}
|{\cal P}^n_2| & \le c(\underline r,\overline r, \|\nabla\bV\|_{L^\infty(Q_T;\Rm^9)}) \times
\\ &
\sum_{K\in {\cal T}}\sum_{\sigma\in {\cal E}(K)}|\sigma| h
\Big(E^{1/2}(\vr_\sigma^{n,{\rm up}}|\hat r_\sigma^{n,{\rm up}})+ E^{2/3}(\vr_\sigma^{n,{\rm up}}|\hat r_\sigma^{n,{\rm up}}\Big)\,|\hat\bV^{n,{\rm up}}_{h,0,\sigma}|\,|{\hat\bV}_{h,0,\sigma}^{n,{\rm up}}   -\hat\bu_\sigma^{n,{\rm up}}|
\\
& \le c(\underline r,\overline r, \|(\bV,\nabla\bV)\|_{L^\infty(Q_T;\Rm^{12})})\Big[\Big(\sum_{K\in {\cal T}}\sum_{\sigma\in {\cal E}(K)}|\sigma| h \Big(E(\vr_\sigma^{n,{\rm up}}|\hat r_\sigma^{n,{\rm up}})\Big)^{1/2}
\\
&
+
\Big(\sum_{K\in {\cal T}}\sum_{\sigma\in {\cal E}(K)}|\sigma| h
E(\vr_\sigma^{n,{\rm up}}|\hat r_\sigma^{n,{\rm up}})\Big)^{2/3}\Big]
 \times
\Big(\sum_{K\in {\cal T}}\sum_{\sigma\in {\cal E}(K)}|\sigma| h\Big|{\hat\bV}_{h,0,\sigma}^{n,{\rm up}}   -\hat\bu_\sigma^{n,{\rm up}}\Big|^6\Big)^{1/6},
\end{aligned}
\end{equation*}
provided $\gamma\ge 3/2$.
Next,  we observe that the contribution of the face $\sigma=K|L$ to the  sums
$
\sum_{K\in {\cal T}}$ $\sum_{\sigma\in {\cal E}(K)}$ $|\sigma| h
E(\vr_\sigma^{n,{\rm up}}|\hat r_\sigma^{n,{\rm up}})$
and $ \sum_{K\in {\cal T}}\sum_{\sigma\in {\cal E}(K)}|\sigma| h|{\hat\bV}_{h,0,\sigma}^{n,{\rm up}}   -\hat\bu_\sigma^{n,{\rm up}}|^6$
is less or equal than
$2|\sigma| h ( E(\vr_K^{n}|\hat r_K^{n}) + E(\vr_L^{n}|\hat r_L^{n})),
$ and
than $2|\sigma| h (|{\bV}_{h,0,K}^{n}   -\bu_K^{n}|^6+ |{\bV}_{h,0,L}^{n}   -\bu_L^{n}|^6)$, respectively. Consequently,we get
by the same reasoning
as in the previous step, under assumption $\gamma\ge 3/2$,
\begin{equation}\label{cP2}
|{\cal P}_2|\le c(\delta, M_0, E_0, \underline r,\overline r, \|(\bV,\nabla\bV)\|_{L^\infty(Q_T;\Rm^{12})})\,\deltat \sum_{n=1}^m {\cal E}(\vr^n,\hat\bu^n|
\hat r^n,\hat \bV_{h,0}^n) + \delta \deltat\sum_{n=1}^m\sum_{K\in {\cal T}}\| \nabla({\bV}_{h,0}^{n}   - \bu^{n})\|^2_{L^2(K;\Rm^3)}.
\end{equation}
}
{\bf Step 3:} {\it Term ${\cal P}_3$.} We realize that
$$
p(\vr_K^n)-p'(r_K^n)(\vr_K^n-r_K^n)-p(r_K^n)\le c(\underline r,\overline r) E(\vr_K|r_K),
$$
by virtue of  Lemma \ref{LL1} in combination with assumption (\ref{i4}).
Consequently,
\begin{equation}\label{cP3}
|{\cal P}_{3}|\le c\|\dv \bV\|_{L^\infty(Q_T)}\deltat \sum_{n=1}^m{\cal E}(\vr^n,\hat\bu^n|\hat r^n,\hat \bV_{h,0}^n).
\end{equation}
\medskip\noindent{\bf Step 4:}{\it Term ${\cal P}_4$.}
We write $\bu^n-\bV^n$ as the sum $(\bu^n-\bV^n_{h,0})+ (\bV^n_{h,0}-\bV^n)$ accordingly splitting
${\cal P}_4$ into two terms
$$
\deltat \sum_{n=1}^m\sum_{K\in {\cal T}}\int_{K} p'(\hat r^n)\frac {\hat r^n-\vr^n}{\hat r^n}(\vc u^n-\bV^n_{h,0})\cdot\nabla r^n\dx\quad\mbox{and}\quad \deltat \sum_{n=1}^m\sum_{K\in {\cal T}}\int_{K} p'(\hat r^n)\frac {\hat r^n-\vr^n}{\hat r^n}(\bV^n_{h,0}-\vc V^n)\cdot\nabla r^n\dx.
$$
Reasoning similarly as in Step 2, we get
\begin{equation}\label{cP4}
\begin{aligned}
|{\cal P}_{4}|& \le h^2\;c(\delta,M_0, E_0,\underline r, \overline r, |p'|_{C([\underline r,\overline r])} \|(\nabla r,\nabla\bV)\|_{L^\infty(\Omega;\Rm^9)})
\\ & + c(\delta, \|\underline r, \overline r, |p'|_{C([\underline r,\overline r])} \|\nabla r\|_{L^\infty(\Omega;\Rm^3)})\;
\deltat \sum_{n=1}^m{\cal E}(\vr^n,\hat\bu^n|\hat r^n,\hat \bV_{h,0}^n) +\delta \,\deltat \sum_{n=1}^m\sum_{K\in {\cal T}}\| \nabla({\bV}_{h,0}^{n}   - \bu^{n})\|^2_{L^2(K;\Rm^3)}.
 \end{aligned}
\end{equation}
Gathering the formulae (\ref{relativeenergy-1}) and (\ref{cP1})-(\ref{cP4}) with $\delta$ sufficiently small (with respect to $\mu$), we conclude the proof of Lemma \ref{Gronwall}.
\end{proof}

\section{{ End}  of the proof of  the error estimate (Theorem \ref{TM1})}

{ Finally, Lemma \ref{Gronwall} in combination with the bound (\ref{est4}) yields
\[{\cal E}(\vr^m,\hat\bu^m|\hat r^m,\hat\bV_{h,0}^m)\le c\Big[h^A+\sqrt{\deltat}+\deltat + {\cal E}(\vr^0,\hat\bu^0|\hat r(0),\hat\bV_{h,0}(0))\Big] + c\,\deltat\sum_{n=1}^{m-1} {\cal E}(\vr^n,\hat\bu^n|\hat r^n,\hat\bV_{h,0}^n);
\]
whence by the discrete standard version of the Gronwall lemma one gets at the first step
$$
{\cal E}(\vr^m,\hat\bu^m|\hat r^m,\hat\bV_{h,0}^m)\le c\Big[h^a+\sqrt{\deltat} + {\cal E}(\vr^0,\hat\bu^0|\hat r(0),\hat\bV_{h,0}(0))\Big].
$$
Going with this information back to  Lemma \ref{Gronwall}, one gets finally
\begin{equation}\label{aa}
{\cal E}(\vr^m,\hat\bu^m|\hat r^m,\hat\bV_{h,0}^m)+
\deltat\frac \mu 2\sum_{n=1}^m\sum_{{K}\in {\cal T}}\int_K|\Grad(\vc u^n-\vc V^n_{h,0})|^2{\rm d x}
\le c\Big[h^a+\sqrt{\deltat} + {\cal E}(\vr^0,\hat\bu^0|\hat r(0),\hat\bV_{h,0}(0))\Big].
\end{equation}

Now, we write
$$
\vr_K^n(\bu^n_K-\vc V_{h,0,K}^n)^2=\vr_K^n(\bu^n_K-\vc V^n)^2+ 2\vr_K^n\vc V^n (\bu^n_K-\vc V_{h,0,K}^n) +
\vr_K^n (\vc V^n-\vc V_{h,0,K}^n)^2,
$$
where
$$
\|\vc V^n-\vc V_{h,0,K}^n\|_{L^\infty(K;\Rm^3)}\aleq
\|\vc V^n-\vc V_{h}^n\|_{L^\infty(K;\Rm^3)}+ \|\vc V_h^n-\vc V_{h,K}^n\|_{L^\infty(K;\Rm^3)}+ \|[\vc V_{h}^n-\vc V_{h,0}^n]_K\|_{L^\infty(K;\Rm^3)}
$$
$$
\aleq h\Big(\|\Grad \vc V^n\|_{L^\infty(K;\Rm^9)}+ \|\Grad \vc V_h^n\|_{L^\infty(K;\Rm^9)}+\|\vc V_{h}^n-\vc V_{h,0}^n\|_{L^\infty(K;\Rm^3)}\aleq h\|\nabla\vc V^n\|_{L^\infty(K;\Rm^9)}.
$$
In the above calculation we have employed formula (\ref{L1-2}) to estimate the first term,
estimates (\ref{L2-1})$_{s=1}$, (\ref{L1-3})$_{s=1}$ to estimate the second term, and formulas
(\ref{N2-2}) and (\ref{N2-3}) for $K\cap\partial\Omega_h=\emptyset$ and $K\cap\partial\Omega_h\neq\emptyset$, respectively,
to evaluate the last term. We conclude that
\begin{equation}\label{L1}
\sum_{K\in {\cal T}}\frac 12{|K|}\Big(\vr^m_K|{\bu}^m_K-{\bV}_{h,0,K}^m|^2-\vr^{0}_K|{\bu}^{0}_K-{\bV}_{h,0,K}^{0}|^2\Big)
\end{equation}
$$
 \ge
\int_{\Omega\cap\Omega_h}\vr^m(\hat\bu^m-\vc V^m)^2{\rm d} x -\int_{\Omega\cap\Omega_h}\vr^0(\hat\bu^0-\vc V^0)^2{\rm d}x + L_{1},
$$
where
$$
|L_{1}|\aleq h\; M_0\|\Grad \vc V\|_{L^\infty((0,T)\times\Omega;\Rm^9)}.
$$

Similarly, we find with help of (\ref{est4}),
$$
\|E(\vr^n_K|\hat r^n)-E(\vr^n_K, r^n)\|_{L^\infty(K)}\le h\;c(M_0,\underline r,\overline r, |p|_{C^1[\underline r,\overline r]}\|\nabla r\|_{L^\infty(Q_T;\Rm^3)});
$$
whence
\begin{equation}\label{L2}
\sum_{K\in {\cal T}}|K|\Big(E(\vr^n_K|\hat r^n)- E(\vr^0_K|\hat r^0) \ge
\int_{\Omega\cap\Omega_h}E(\vr^m| r^m){\rm d}x-\int_{\Omega\cap\Omega_h}E(\vr^0| r^0){\rm d}x + L_2,
\end{equation}
where
$$
|L_2|\le h\;c(M_0,\underline r,\overline r, |p|_{C^1[\underline r,\overline r]},\|\nabla r\|_{L^\infty(Q_T;\Rm^3)}).
$$
 Finally, by virtue of (\ref{N2-2}--\ref{N2-3}) and (\ref{L1-3})$_{s=2}$
 $$
 \|\nabla (\bV^n_{h,0}-\bV^n)\|_{L^2(K;\Rm^3)}\aleq  h \|(\nabla\bV^n,\nabla^2\bV^n)\|_{L^\infty(K;\Rm^{12})};
 $$
 whence
 \begin{equation}\label{L3}
 \deltat\sum_{n=1}^m\sum_{{K}\in {\cal T}}\int_K|\Grad(\vc u^n-\vc V^n_{h,0})|^2{\rm d x} \ge\deltat\sum_{n=1}^m
 \int_{\Omega \cap\Omega_h}|(\nabla_h\vc u^n-\Grad\vc V^n)|^2{\rm d x} + L_3,
\end{equation}
where
$$
|L_3| \le h^2 c( \|(\nabla\bV^n,\nabla^2\bV^n)\|_{L^\infty(K;\Rm^{12})}).
$$

Theorem \ref{TM1} is a direct consequence of estimate (\ref{aa}) and identities (\ref{L1}--\ref{L3}). Theorem \ref{TM1} is thus proved.}


\section{Concluding remarks}

In the convergence proofs one usually needs to complete the numerical scheme by stabilizing terms, so that
the new numerical scheme reads

\bFormula{num2+}
\sum_{K \in \mathcal{T}_h } |K| \frac{ \vr^n_K - \vr^{n-1}_K }{ \Det } \phi_K + \sum_{K \in \mathcal{T}_h }
\sum_{ \sigma \in \mathcal{E}(K) } |\sigma| \vr^{n,{\rm up}}_\sigma ( \vu^n_{\sigma} \cdot \vc{n}_{\sigma, K} ) \phi_K + T_c(\phi)= 0,
\eF
$$
\mbox{for any}\ \phi \in Q_h(\Omega_h)\ \mbox{and}\ n=1,\ldots,N,
$$
\bFormula{num3+}
\sum_{K\in{\cal T}}\frac{|K|}{\deltat} \Big({\vr^n_K{{\bu}}^n_{K}- \vr^{n-1}_K{{\bu}}^{n-1}_{K}} \Big)\cdot \bv_K+ \sum_{K\in {\cal T}}\sum_{\sigma \in {\cal E}(K)} |\sigma|\vr^{n,{\rm up}}_\sigma {{ \hat\bu}}_{\sigma}^{n,{\rm up}}[\bu^n_\sigma\cdot \bn_{\sigma,K}]\cdot \bv_K
\eF
\[
- \sum_{K\in{\cal T}}p(\vr^n_K)\sum_{\sigma \in {\cal E}(K)}  |\sigma|\bv_\sigma\cdot {\vc n}_{\sigma,K}+\mu\sum_{K\in{\cal T}}\int_K \nabla\bu^n : \nabla\bv \  \dx
\]
\[
+ \frac{\mu}{3} \sum_{K\in{\cal T}}\int_K{\rm div}\bu^n{\rm div}\bv\,  \dx + T_m(\phi)=0,
 \ \mbox{for any}\ \bv \in {V_{h,0}(\Omega;R^3) }\ \mbox{and}\ n=1,\ldots,N,
\]
where
$$
T_c(\phi)= h^{1-\ep}\sum_{\sigma\in {\cal E}_{\rm int}}|\sigma|[\vr^n]_{\sigma,\bn_\sigma}[\phi]_{\sigma\bn _\sigma},\quad
T_m(\phi)= \sum_{\sigma\in {\cal E}_{\rm int}}|\sigma|[\vr^n]_{\sigma,\bn_\sigma}\{\hat\bu ^n\}_{\sigma}[\hat\phi]_{\sigma,\bn_\sigma},\;\ep\in [0,1),
$$
see Karlsen, Karper \cite{KK}, Gallouet, Gastaldo, Herbin, Latch\'e \cite{GGHL2008baro}. These terms are designed to provide the supplementary positive term
$$
h^{1-\ep}\sum_{\sigma\in {\cal E}_{\rm int}}|\sigma|[\vr^n]_{\sigma,\bn_\sigma}^2,
$$
to the left hand side of the discrete energy identity (\ref{denergyinequality}). They contribute to the right hand side
of the discrete relative energy (\ref{drelativeenergy}) by supplementary terms whose absolute value is bounded from above by
$$
h^{(1-\ep)/2}\;c\Big(M_0,E_0, \sup_{n=0,\ldots,N}\|r^n,\vc U^n,\nabla\vc U^n\|_{L^\infty(\Omega_h;\Rm^{13})}, \sup_{n=0,\ldots,N}
\sup_{\sigma\in{\cal E} {\rm int}}[r^n]_{\sigma,\bn_\sigma}/h\Big).
$$
Consequently, they give rise to the contributions at the right hand side of the approximate relative energy inequality  (\ref{relativeenergy-}) whose  bound is
$$
h^{(1-\ep)/2}\;c\Big(M_0,E_0, \|r,\nabla r,\vc U,\nabla\vc U\|_{L^\infty(Q_T;\Rm^{16})}
\Big).
$$

Similar estimates are true, if we replace in the numerical scheme everywhere classical upwind formula (\ref{upwind1})
\[ {\rm Up}_K(q,\vu)=
\sum_{\sigma\in{\cal E}(K)} q_\sigma^{\rm up}{\bu}_\sigma\cdot{\vc n}_{\sigma,K}=\stik \Big(q_K [{\bu}_\sigma\cdot{\vc n}_{\sigma,K}]^+ + q_L [{\bu_\sigma}\cdot{\vc n}_{\sigma,K}]^-\Big),
\]
by the modified upwind
suggested in \cite{FMK}:
\begin{equation} \label{upwind2}
{\rm Up}_K(q,\vu)= \stik
                      \frac{q_K}{2}\Big([\bu_\sigma\cdot\bn_{\sigma,K}+h^{1-\ep}]^+ +  [\bu_\sigma\cdot\bn_{\sigma,K}-h^{1-\ep}]^+\Big)+
                      \frac{q_L}{2}\Big([\bu_\sigma\cdot\bn_{\sigma,K}+h^{1-\ep}]^- +  [\bu_\sigma\cdot\bn_{\sigma,K}-h^{1-\ep}]^-\Big),
\end{equation}
where $\sigma=K|L\in {\cal E}_{\rm int}$. We will finish by formulating the error estimate for the numerical problem (\ref{num1}), (\ref{num2+}), (\ref{num3+}) or for (\ref{num1}), (\ref{num2}), (\ref{num3})  with modified upwind
(\ref{upwind2}).
\bTheorem{M2}
Let $\Omega$, $p$, $[r_0,\vc V^0]$, $[r,V]$ satisfy assumptions of Theorem \ref{TM1}. Let $(\vr^n,\vc u^n)_{n=0,\ldots,N}$ be a family of numerical solutions to the scheme (\ref{num1}), (\ref{num2+}), (\ref{num3+})
or to the scheme (\ref{num1}), (\ref{num2}), (\ref{num3})  with modified upwind
(\ref{upwind2}), where $\ep\in [0,1)$. Then error estimate ({\ref{M1}) holds true with the exponent
$$
a=\min\Big\{\frac{2\gamma-3}\gamma,\frac {1-\ep} 2\Big\}\;\mbox{if $\frac 32\le \gamma< 2$},\quad
a= \frac {1-\ep} 2\;\mbox{if $\gamma\ge 2$}.
$$
}
\eT


\def\cprime{$'$} \def\ocirc#1{\ifmmode\setbox0=\hbox{$#1$}\dimen0=\ht0
  \advance\dimen0 by1pt\rlap{\hbox to\wd0{\hss\raise\dimen0
  \hbox{\hskip.2em$\scriptscriptstyle\circ$}\hss}}#1\else {\accent"17 #1}\fi}

\end{document}